\documentclass[letterpaper,11pt]{amsart}
\usepackage{amsmath,amssymb,amsthm,pinlabel,tikz,hyperref,mathrsfs,color, thmtools}
\usepackage{verbatim}
\usepackage{float}
\usepackage{caption}
\usepackage{subcaption}
\usepackage{enumitem}
\newcommand{\nc}{\newcommand}
\nc{\dmo}{\DeclareMathOperator}
\dmo{\ra}{\rightarrow}
\dmo{\Prob}{\mathbb{P}}
\dmo{\E}{\mathbb{E}}
\dmo{\N}{\mathbb{N}}
\dmo{\Z}{\mathbb{Z}}
\dmo{\Q}{\mathbb{Q}}
\dmo{\R}{\mathbb{R}}
\dmo{\C}{\mathcal{C}}
\dmo{\X}{\mathcal{X}}
\dmo{\U}{\mathcal{U}}
\dmo{\T}{\mathcal{T}}
\dmo{\F}{\mathcal{F}}
\dmo{\AC}{\mathcal{AC}}
\dmo{\w}{\omega}
\dmo{\MIN}{\mathcal{MIN}}
\dmo{\Mod}{Mod}
\dmo{\PMod}{PMod}
\dmo{\PMF}{\mathcal{PMF}}
\dmo{\Mat}{Mat}
\dmo{\supp}{supp}
\dmo{\UE}{\mathcal{UE}}
\dmo{\vol}{vol}
\dmo{\B}{B}
\dmo{\PB}{PB}
\dmo{\PR}{PSL(2,\mathbb{R})}
\dmo{\GL}{GL(k, \mathbb{C})}
\dmo{\SL}{SL(2, \mathbb{Z})}
\dmo{\Isom}{Isom}
\dmo{\RP}{\mathbb{R} \mathrm{P}}
\dmo{\I}{\mathcal{I}}
\dmo{\el}{\ell_{\C}}
\dmo{\NN}{\mathcal{N}}
\dmo{\rk}{rank}
\dmo{\tr}{tr}
\dmo{\llangle}{\langle\langle}
\dmo{\rrangle}{\rangle\rangle}
\dmo{\Unif}{Unif}
\dmo{\Out}{Out}
\dmo{\sumRho}{\mathcal{N}}
\dmo{\stopping}{\vartheta}

\usetikzlibrary{decorations.markings, patterns}
\tikzset{->-/.style={decoration={
  markings,
  mark=at position #1 with {\arrow{>}}},postaction={decorate}}}

\nc{\nt}{\newtheorem}

\nt{theorem}{Theorem}

\newtheorem{thm}{{\bf Theorem}}[section]
\newtheorem{lem}[thm]{{\bf Lemma}}
\newtheorem{cor}[thm]{{\bf Corollary}}
\newtheorem{prop}[thm]{{\bf Proposition}}
\newtheorem{fact}[thm]{Fact}

\newtheorem{claim}[thm]{Claim} 
\newtheorem{remark}[thm]{Remark}

\newtheorem{question}[thm]{Question}

\newtheorem{definition}[thm]{Definition}
\numberwithin{equation}{section}

\title[Genericity of pseudo-Anosovs]{Pseudo-Anosovs are exponentially generic in mapping class groups}

\date{\today}
\author{Inhyeok Choi}

\address{%
		Department of Mathematical Sciences, KAIST\\
		291 Daehak-ro Yuseong-gu, Daejeon, 34141, South Korea 
}
\email{%
        inhyeokchoi@kaist.ac.kr
        }

\begin{document}
\begin{abstract}
Given a finite generating set $S$, let us endow the mapping class group of a closed hyperbolic surface with the word metric for $S$. We discuss the following question: does the proportion of non-pseudo-Anosov mapping classes in the ball of radius $R$ decrease to 0 as $R$ tends to infinity? We show that any finite subset $S'$ of the mapping class group is contained in a finite generating set $S$ such that this proportion decreases exponentially. Our strategy applies to weakly hyperbolic groups and does not refer to the automatic structure of the group.


\noindent{\bf MSC classes:} 20F67, 30F60, 57K20, 57M60, 60G50
\end{abstract}

\maketitle

%
%

\section{Introduction}	\label{sec:introduction}

Let $\Sigma$ be a closed hyperbolic surface. We denote by $\Mod(\Sigma)$, $\T(\Sigma)$ and $\mathcal{C}(\Sigma)$ the mapping class group, the Teichm{\"u}ller space and the curve complex of $\Sigma$, respectively. When $X$ is a Gromov hyperbolic space or $\T(\Sigma)$ and $g \in \textrm{Isom}(X)$, we denote by $\tau_{X}(g)$ the (asymptotic) translation length of $g$. For a group $G$ generated by a finite set $S$, we denote by $B_{S}(n)$ the ball of radius $n$ with respect to the word metric for $S$. We also denote by $\partial B_{S}(n)$ the corresponding sphere of radius $n$. Our main result is as follows.

\begin{theorem}[Translation length grows linearly]\label{thm:generic}
Let $X$ be either a Gromov hyperbolic space or $\T(\Sigma)$. Let also $G$ be a finitely generated non-elementary subgroup of $\Isom(X)$ and $S' \subseteq G$ be a finite subset. Then there exist $L, K>0$ and a finite generating set $S \supseteq S'$ of $G$ such that \[
\frac{\#\{g \in B_{S}(n) : \tau_{X}(g) \le Ln\}}{\#B_{S}(n)} \le K e^{-n/K}
\]
holds for each $n$.
\end{theorem}

Non-pseudo-Anosov mapping classes have translation length zero in $\C(\Sigma)$. As a result, we affirmatively answer the following version of a folklore conjecture, at least for infinitely many generating sets $S$.

\begin{cor}[Genericity of pAs, {cf. \cite[Conjecture 3.15]{farb2006problems}}]
Let $G$ be a finitely generated non-elementary subgroup of $\Mod(\Sigma)$. Then there exists a finite generating set $S \subseteq G$ such that the proportion of non-pseudo-Anosov mapping classes in the ball $B_{S}(n)$ decays exponentially as $n \rightarrow \infty$.
\end{cor}

Note that $\Mod(\Sigma)$ can act on both $\T(\Sigma)$ and $\mathcal{C}(\Sigma)$. Comparing the translation lengths of mapping classes on these two spaces is an interesting question. Thanks to the linear growth in Theorem \ref{thm:generic}, we can deduce:

\begin{cor}\label{cor:generic}
Let $G$ be a finitely generated non-elementary subgroup of $\Mod(\Sigma)$ and $S' \subseteq G$ be a finite subset. Then there exist $L, K>0$ and a finite generating set $S \supseteq S'$ of $G$ such that the following holds. For \[
P:= \left\{ g \in G : \frac{1}{L} \le \frac{\tau_{\T(\Sigma)}(g) }{\tau_{\mathcal{C}(\Sigma)}(g)}\le L \right\},
\]
and each $n$, we have
\[
\frac{\#B_{S}(n) \cap P^{c}}{\#B_{S}(n)} \le K e^{-n/K}.
\]
\end{cor}

\subsection{History and related problems}

We remark that Theorem \ref{thm:generic} may not be optimal for Gromov hyperbolic spaces. Gekhtman, Taylor and Tiozzo have proved the genericity of loxodromics in hyperbolic groups acting on separable Gromov hyperbolic spaces, in terms of the word metric for \emph{any} finite generating set. This was generalized to relatively hyperbolic groups, RAAGs and RACGs with \emph{particular} finite generating sets (see \cite{gekhtman2018counting}, \cite{gekhtman2020graph}, \cite{gekhtman2020clt}). Gekhtman-Taylor-Tiozzo's example \cite[Example 1]{gekhtman2020graph} shows that the genericity of loxodromics is not achieved for all weakly hyperbolic groups with respect to all finite generating sets.

For relatively hyperbolic groups (and many more), Yang also established that loxodromics are exponentially generic when the action is proper and cocompact \cite{yang2020genericity}. Hence, at least for hyperbolic groups that admit a proper and cocompact action on Gromov hyperbolic spaces, Theorem \ref{thm:generic} is weaker than previous results in the sense that the finite generating set cannot be arbitrary. It is however stronger in the sense that \begin{enumerate}
\item it does not require the action of $G$ to be proper and cocompact,
\item it deals with exponential genericity with respect to a linearly growing threshold, not a static threshold.
\end{enumerate}
In fact, combining our strategy with the theory of Gekhtman, Taylor and Tiozzo yields the following.

\begin{prop}\label{prop:genericGekht}
Let $X$, $G$ be as in Theorem \ref{thm:generic} and $S$ be a finite generating set of $G$. Suppose moreover that $G$ itself is a hyperbolic group. Then there exists $\lambda >0$ such that the following hold. Below, $\nu$ denotes the Patterson-Sullivan measure with respect to $S$.\begin{enumerate}
\item For any $x \in X$ and $\nu$-a.e. $\eta \in \partial G$, if $(g_{n})_{n \ge 0}$ is a geodesic in $G$ converging to $\eta$, then \[
\lim_{n \rightarrow \infty} \frac{d_{X}(x, g_{n}x)}{n} = \lambda.
\]
\item For any $\epsilon > 0$, there exists $K>0$ such that \[
\frac{\#\{g \in \partial B_{S}(n) : \tau_{X}(g) \notin [\lambda-\epsilon, \lambda+\epsilon]\}}{\#\partial B_{S}(n)} \le K e^{-n/K}
\]
holds for each $n$.
\end{enumerate}
\end{prop}

Note that this implies the following. For any hyperbolic subgroup $G$ of $\Mod(\Sigma)$ and \emph{any} finite generating set $S$ of $G$, let $\lambda_{\T}$, $\lambda_{\C}$ be the escape rate of $G$ on $\T(\Sigma)$ and $\C(\Sigma)$, respectively, in terms of the Patterson-Sullivan measure for $S$. Then for any $\epsilon>0$, there exists $K$ such that \[
\frac{\#\{g \in B_{S}(n) : \frac{\lambda_{\T}}{\lambda_{\C}} - \epsilon \le \frac{\tau_{\T(\Sigma)}(g) }{\tau_{\mathcal{C}(\Sigma)}(g)}\le  \frac{\lambda_{\T}}{\lambda_{\C}} +\epsilon\}}{\#B_{S}(n)} \le K e^{-n/K}
\]
holds for each $n$. For the sake of completeness, we sketch the proof of Proposition \ref{prop:genericGekht} in Appendix \ref{section:appendix}.

Meanwhile, Theorem \ref{thm:generic} is new for $\Mod(\Sigma)$. The progress so far was that the proportion of pseudo-Anosov elements in the word metric ball stays bounded away from zero \cite{cumplido2018pA}. See \cite{yang2020genericity} and \cite{erlandsson2020pA} for counting pseudo-Anosovs in other viewpoints. We can further ask:

\begin{question}
Are pseudo-Anosovs exponentially generic with respect to any finite generating set? For example, are they exponentially generic with respect to Humphries' generators? If not, are they generic at least?
\end{question}

\begin{question}\label{question:Mod}
Does Proposition \ref{prop:genericGekht} hold for $G=\Mod(\Sigma)$ and at least one $S$?
\end{question}

Question \ref{question:Mod} is intimately related to the (geodesic) automaticity of $\Mod(\Sigma)$.

Let us finally mention a problem investigated by I. Kapovich. Let $\mu$ be a discrete measure on a group $G$. We define the \emph{non-backtracking random walk} generated by $\mu$ as follows. The first alphabet $g_{1}$ is chosen from $G$ with the law of $\mu$; for each $n \ge 2$, $g_{n}$ is chosen from $G \setminus \{g_{n-1}^{-1}\}$ with the law \[
\Prob(g_{n} = g) = \frac{1}{\mu(G \setminus \{g_{n-1}^{-1}\})} \mu(g).
\]
In this setting, Gekhtman, Taylor and Tiozzo proved that $\Prob(\w_{n}\,\,\textrm{is loxodromic})$ tends to 1 as $n \rightarrow \infty$ \cite[Theorem 2.8]{gekhtman2020graph}. With an adequate modification, our argument yields the following. 

\begin{prop}\label{prop:genericKapo}
Let $X$ be as in Theorem \ref{thm:generic}, and $\mu$ be a non-elementary discrete measure $\mu$ on $\textrm{Isom}(X)$. Consider the non-backtracking random walk $\w$ generated by $\mu$. Then there exists $L, K>0$ such that \[
\Prob\{ \tau_{X}(\w_{n}) \le Ln \}\le K e^{-n/K}
\]
holds for each $n$.
\end{prop}

\subsection{Strategy for Theorem \ref{thm:generic}}

 One approach to the counting problem is to utilize (geodesic) automatic structures of the group. Lacking such structures, we instead consider the random walk $\w$ on $G$ generated by the uniform measure on $S$. Then Gou{\"e}zel's and Baik-Choi-Kim's theories imply that the unwanted probability decays exponentially (\cite{gouezel2021exp}, \cite{baik2021linear}). There are at least two more theories that provide this exponential decay. One is Maher's theory \cite{maher2012exp} for random walks with bounded support, which led to Maher-Tiozzo's more general theory \cite{maher2018random}. Another one is Boulanger-Mathieu-Sert-Sisto's theory \cite{boulanger2020large} of large deviation principles for random walks with finite exponential moment.

Unavoidably, random walks cannot count lattice points in a one-to-one manner. If $S$ is nicely populated by the self-convolution of a Schottky set $S_{0}$, however, then we have a one-to-one correspondence between some portion of lattice points and (non-backtracking) paths of alphabets in $S_{0}$. This leads to the estimate $\#B_{S}(e, n) \sim \#(\textrm{paths from the random walk})\cdot r^{n}$ with $r \sim 1$. We arrive at the desired estimate by forcing the exponential decay of probability to be much faster than the decay of $r^{n}$.

Let us bring a toy example to explain how this strategy works. Let $S$ be a finite symmetric generating set of the free group $F_{2} \simeq \langle a, b\rangle$ of rank 2. Our goal is to compare the growth rate of \[\begin{aligned}
A(n) &:= \{a_{1} \cdots a_{n} : a_{i} \in S\}, \\
B(n) &:= \{a_{1} \cdots a_{n}: a_{i} \in S, \tau_{X}(a_{1} \cdots a_{n}) = 0\}.
\end{aligned}
\]
Here, $S$ contains elements that cancel out each other. This implies that although $A(n)$ does grow exponentially, its growth rate may not equal $\#S$. Moreover, even though any nontrivial word in $F_{2}$ corresponds to a loxodromic isometry on the Cayley graph, there can be some sequences of $n$ letters from $S$ whose composition is trivial.

These concerns disappear if $S$ were an alphabet for a free subsemigroup of $F_{2}$. Letters in $S$ do not cancel out each other, and we have $A(n) \sim (\#S)^{n}$ and $B(n) = 0$ for $n \ge 1$. The contrast between the two growth rates persists even when a few letters of $S$ cancel out each other. For example, let us take \[
S = \{ a^{2}, ab, ba, a^{-2}, b^{-1}a^{-1}, a^{-1} b^{-1}\}.
\] Clearly $a^{2}$ and $a^{-2}$, $ab$ and $b^{-1}a^{-1}$, and $ba$ and $a^{-1}b^{-1}$ cancel out each other. Nonetheless, $A(n)$ still grows exponentially with the growth rate $\#S - 1$. On the other hand, the growth rate of $B(n)$ is $\#S$ times the spectral radius of the simple random walk on the homogenous tree of degree 6, which is $\frac{\sqrt{5}}{3}$. The situation gets better and better as $S$ becomes larger and larger. For a symmetric free generating set $S$ of a subgroup $G$ of $F_{2}$ with $\operatorname{rk}(G) = d$, we have $A(n) \sim (2d-1)^{n}$ while $B(n) \sim \left(\frac{2\sqrt{2d-1}}{2d} \cdot 2d\right)^{n}$. As $d$ increases, the growth rate of $A(n)$ becomes closer to $\#S$ while the growth rate of $B(n)$ stays uniformly strictly smaller than $\#S$. In summary, although cancellations may disturb the contrast between the two growth rates, such disturbance can be made negligible by taking an `almost mutually independent' set $S$. Such sets are called Schottky sets.

Let us now consider a choice $S= S_{1} \cup S_{0}$, where $S_{1}$ is a (symmetric) alphabet for a free semigroup of $F_{2}$ and $S_{0}$ is an additional impurity that makes $S$ a generating set of $F_{2}$. For simplicity let us assume the form \[
S_{1} = \left\{ c_{1} \cdots c_{M} \cdot a^{100M} \cdot d_{1} \cdots d_{M} : c_{i}, d_{i} \in \{a, b\}\right\}
\]
for $M \ge 10$. Since $S_{1}$ is large enough, we can guarantee the gap between the growth rates of $A(n)$ and $B(n)$ with respect to $S_{1}$. We now want the stability of this gap with respect to the perturbation $S_{0}$. To be explicit, we hope \[
A(n) \gtrsim ((1-\rho) \#S)^{n}, \quad B(n) \lesssim (\rho'  \#S)^{n}
\]
for some $\rho$ proportional to $(\#S_{0})/ (\#S)$ and a constant $\rho' <1$ that works for small enough $(\#S_{0})/(\#S)$. Note that we do not require any condition on the elements of $S_{0}$; we only restrict the cardinality of $S_{0}$. 

The stability for $A(n)$ is straightforward, but the one for $B(n)$ is considerably harder. In contrast to the case of simple random walks, we have no information how individual elements of $S_{0}$ interact with the elements of $S_{1}$. Perhaps a bad element of $S_{0}$ cancels out a concatenation of 10 letters in $S_{1}$, or that of 100 letters. This opens the possibility that the RV $d(o, \w_{n+1} o) - d(o, \w_{n} o)$ conditioned on $\w_{n}$ has negative expectation for some $\w_{n}$. Hence, we cannot pretend as if we are summing up i.i.d. RVs with positive expectation at each single step. 

Nonetheless, we may focus only on the steps chosen from $S_{1}$ and try to construct i.i.d. RVs that reflects the progresses made there. Let us construct a set $P_{n} = \{i(1) < \ldots < i(m)\}$ such that: \begin{enumerate}
\item $g_{i(1)}, \ldots, g_{i(m)}$ are drawn from $S_{1}$, and 
\item $[o, \w_{n} o]$ contains the middle 99\% of $[\w_{i(1)} o, \w_{i(1)} g_{i(1)} o]$, $\ldots$, $[\w_{i(m)} o, \w_{i(m)} g_{i(m)} o]$.
\end{enumerate}
Note that $P_{n}$ is a subset of \[
\Theta = \{1 \le i \le n : g_{i} \in S_{1}\},
\]
whose size is sufficiently large if $(\#S_{0})/(\#S)$ is small enough. Fixing the slots $\Theta$ for elements in $S_{1}$ and all the other choices $\{g_{i} : i \notin \Theta\}$, we are now asked to control $w_{m} = w_{0} s_{1} w_{1} \cdots s_{m} w_{m}$ where $w_{i}$ are fixed words in $F_{2}$ and $s_{i}$ are independently drawn from $S_{1}$. Since elements of $S_{1}$ are deviating from each other early, $[o, so]$ and $[o, w_{0}^{-1}o]$ fellow travel for very few $s \in S_{1}$. Similarly, $[o, s^{-1} o]$ and $[o, w_{1} o]$ are deviating early for a large probability. Due to these sorts of reasons, there is a (uniformly) high chance that the middle 99\% of $[w_{k-1} o, w_{k-1} s_{k}o]$ is visible in $[o, w_{k-1} s_{k} w_{k} o]$. Consequently, the progress made by $s_{k}$ is along $[o, w_{0} s_{1} \cdots s_{k} w_{k} o]$. In such a case, those progresses made by $s_{i}$ for $i < k$ along $[o, w_{0} s_{1} \cdots w_{k-1} o]$ are still intact in $[o, w_{0} s_{1} \cdots s_{k} w_{k} o]$.

Still, we are worried with the situation that an unfortunate choice of $s_{k}$ makes a progress that is not visible in $[o, w_{0} \cdots s_{k} w_{k} o]$, and even worse, the previous Schottky progresses along $[o, w_{0} \cdots s_{k-1} w_{k-1} o]$ are all lost in $[o, w_{0} \cdots s_{k} w_{k} o]$. Let $s_{j(1)}, \ldots, s_{j(m')}$ be the choices from $S_{1}$ before $s_{k}$ that made progresses along $[o, w_{0} \cdots s_{k-1} w_{k-1} o]$. We observe that there are plenty of other choices for $s_{j(1)}$, $\ldots$, $s_{j(m')}$ that make progresses along $[o, w_{0} \cdots s_{k-1} w_{k-1} o]$. This modification does not affect the positions $j(1)$, $\ldots$, $j(m')$, and we call it \emph{pivoting}. We now freeze the choice $w_{j(m')+1} s_{j(m')+1} \cdots s_{k} w_{k}$ and perform the pivoting. The progress $[o, s_{j(m')}]$ made by the $m'$-th choice $s_{j(m')}$ is aligned along $[o, s_{j(m')} w_{j(m')} s_{j(m')+1} \cdots s_{k} w_{k} o]$ for high chance, and moreover, $[o, s_{j(m'-1)}]$ is aligned along $[o, s_{j(m'-1)} w_{j'(m'-1)} \cdots w_{k} o]$ for high chance \emph{regardless of the pivotal choice $s_{j(m')}$}. Continuing this, we observe that the progress made by the $(m'-l)$-th choice $s_{j(m'-l)}$ -- and all progresses before it -- is visible in $[o, w_{0} \cdots w_{k} o]$, outside a set of probability that decays exponentially in $l$. Using this, we can bound $\#P_{n}$ from below by the sum of i.i.d. that heavily favors 1 and has an exponentially decaying tail. In particular, we have \[
\Prob(\#P_{n} \ge n/K) \le K e^{-n/K}
\] for some $K>0$. A fuller description can be found in \cite[Section 2]{gouezel2021exp}. 

We have not discussed the linearly growing threshold of the translation length yet. For example, the word $a^{-1}b^{-1} a^{3} b a^{2} b^{-2} a^{-2} b^{-1} a^{-3} ba$ has large displacement, 18, but short translation length, 2. If we pivot the Schottky choices, e.g., modify the first or the second $b^{-1}$ into $b$, its translation length will increase to 16 or 8, respectively. This illustrates that pivoting can secure large translation lengths in probability, which is Baik-Choi-Kim's strategy in \cite{baik2021linear}.

The above example is an over-simplification of weakly hyperbolic groups and mapping class groups. Gromov hyperbolic spaces and Teichm{\"u}ller space are not trees, but we can still copy the above argument by using the Gromov inequality or the (partial) hyperbolicity in  Teichm{\"u}ller space due to Rafi.

We remark that the notions and statements in Section \ref{section:group}, \ref{section:pivot} and \ref{section:trLength} are mostly established in \cite{gouezel2021exp}, \cite{baik2021linear} and \cite{choi2021clt}. Nonetheless, to make the exposition self-contained, we present all details except the proofs of Fact \ref{lem:concat}, \ref{lem:concatUlt}, \ref{lem:farSegment}, \ref{lem:1segment} and \ref{lem:passerBy}. See \cite{choi2021clt} for their proofs.

\subsection*{Acknowledgments}
The author truly thanks Hyungryul Baik, Dongryul M. Kim, Jing Tao, Samuel Taylor and Wenyuan Yang for their helpful discussion. The author also appreciates anonymous referees' comments that led to a significant improvement on the manuscript.

The author was partially supported by Samsung Science \& Technology Foundation grant No. SSTF-BA1702-01.

%
%

\section{Witnessing and Schottky sets}\label{section:group}

This section is devoted to the facts about Gromov hyperbolic spaces and Teichm{\"u}ller spaces. All facts in this section are proved in \cite{choi2021clt}.

Given a metric space $(X, d)$ and a triple $x, y, z \in X$, we define the \emph{Gromov product} of $y, z$ with respect to $x$ by \[
(y, z)_{x} = \frac{1}{2} [d(x, y) + d(x, z) - d(y, z)].
\]
Throughout the article, we will frequently use the following property: for $x, y, z \in X$ and $g \in \operatorname{Isom}(X)$, we have \[
(gy, gz)_{gx} = (y, z)_{x}.
\]

$X$ is said to be \emph{Gromov hyperbolic} if there exists a constant $\delta>0$ such that every quadruple $x, y, z, w \in X$ satisfies the following inequality called the \emph{Gromov inequality}:
\begin{equation}\label{eqn:Gromov}
(x, y)_{w} \ge \min\{(x, z)_{w}, (y, z)_{w}\} - \delta.
\end{equation}
Here, $X$ need not be geodesic nor intrinsic; all arguments regarding Gromov hyperbolic spaces rely solely on the Gromov inequality.

\emph{From now on, we permanently fix $\delta>0$.}

In $\T(\Sigma)$, we say that a surface $x \in \T(\Sigma)$ is $\epsilon$-thin if there exists a simple closed curve on $x$ whose extremal length is less than $\epsilon$; if not, we say that it is $\epsilon$-thick. For $x, y \in \T(\Sigma)$, $[x, y]$ denotes the Teichm{\"u}ller geodesic from $x$ to $y$; it is said to be $\epsilon$-thick if it is composed of $\epsilon$-thick points.
 
In $\delta$-hyperbolic spaces, we regard every point as $\epsilon$-thick for any $\epsilon>0$. Here $[x, y]$ denotes the pair of points $(x, y)$, which is considered $\epsilon$-thick for any $\epsilon>0$. In either space, $[x, y]$ are called \emph{segments} and their lengths are defined by $d(x, y)$.

\begin{definition}[Witnessing in $\delta$-hyperbolic spaces]\label{dfn:witnessing1}
Let $x$, $y$, $\{x_{i}\}_{i=1}^{n}$, $\{y_{i}\}_{i=1}^{n}$ be points in a $\delta$-hyperbolic space $X$, and $D>0$. We say that $[x, y]$ is \emph{$D$-witnessed by $([x_{1}, y_{1}], \ldots, [x_{n}, y_{n}])$} if: \begin{enumerate}
\item $(x_{i-1}, x_{i+1})_{x_{i}} < D$ for $i = 1, \ldots, n$, where $x = x_{0}$ and $y = x_{n+1}$;
\item $(y_{i-1}, y_{i+1})_{y_{i}} < D$ for $i = 1, \ldots, n$, where $x = y_{0}$ and $y = y_{n+1}$, and
\item $(y_{i-1}, y_{i})_{x_{i}}, (x_{i}, x_{i+1})_{y_{i}} < D$ for $i=1, \ldots, n$.
\end{enumerate}
\end{definition}

Definition \ref{dfn:witnessing1} seems complicated, but it is a version of Definition \ref{dfn:witnessing2} in the absence of the geodesicity of the ambient space. Indeed, in a geodesic Gromov hyperbolic space, these two notions of witnessing are equivalent up to the modification of the parameter $D$.

\begin{definition}[Witnessing in $\T(\Sigma)$]\label{dfn:witnessing2}
Let $x$, $y$, $\{x_{i}\}_{i=1}^{n}$, $\{y_{i}\}_{i=1}^{n}$ be points in $X = \T(\Sigma)$, and $D>0$. We say that $[x, y]$ is \emph{$D$-witnessed by $([x_{1}, y_{1}], \ldots, [x_{n}, y_{n}])$} if the geodesic $[x, y]$ contains subsegments $[x'_{i}, y'_{i}]$ such that \begin{enumerate}
\item $x_{i+1}'$ is not closer to $x$ than $y_{i}'$ for $i=1, \ldots, n-1$, and 
\item $[x_{i}, y_{i}]$ and $[x_{i}', y_{i}']$ $D$-fellow travel.
\end{enumerate}
\end{definition}

\emph{From now on, we permanently fix $X$, a $\delta$-hyperbolic space or $\T(\Sigma)$}.

\begin{definition}
Let $x, y$ and $\{x_{i}, y_{i}, z_{i}\}_{i=1}^{N}$ be points in $X$. We say that $[x, y]$ is \emph{$D$-marked with $([x_{i}, y_{i}])_{i=1}^{N}$, $([z_{i}, x_{i}])_{i=1}^{N}$} if the following hold: \begin{enumerate}
\item $(y_{i}, z_{i})_{x_{i}} < D$ for each $i=1, \ldots, N$;
\item $[x_{i-1}, x_{i}]$ is $D$-witnessed by $([x_{i-1}, y_{i-1}], [z_{i}, x_{i}])$ for each $i=1, \ldots, N+1$, where we set $x_{0} = y_{0} = x$ and $x_{N+1} =z_{N+1}= y$;
\end{enumerate}
In this case, we also say that $[x_{1}, x_{n}]$ is fully $D$-marked with $([x_{i}, y_{i}])_{i=1}^{N-1}$, $([z_{i}, x_{i}])_{i=2}^{N}$.
\end{definition}

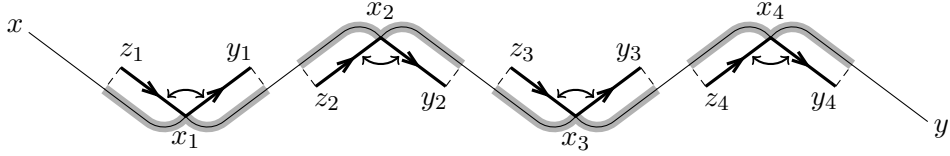
\begin{figure}[H]
\begin{tikzpicture}[scale=0.9]
\def\a{0.8}
\def\b{3}
\def\c{1.2}

\draw[line width = 1.4 mm, black!30, rotate=36.869897645844021]  (0, 0) arc (-180:-90:0.5*\a) -- (\c, -0.5*\a);
\draw[line width = 1.4 mm, black!30, rotate=36.869897645844021]  (\b - \c, -0.5*\a) -- (\b - 0.5*\a, -0.5*\a) arc (90:0:0.5*\a);
\draw[line width = 1.4 mm, black!30, shift={(\a*0.6 + \b*0.8, -\a*0.8 + \b*0.6)}, rotate=-36.869897645844021] (0, 0) arc (180:90:0.5*\a) -- (\c, 0.5*\a);

\draw[line width = 1.4 mm, black!30,rotate=-36.869897645844021] (0, 0) arc (0:-90:0.5*\a) -- (-\c, -0.5*\a);

\begin{scope}[shift={(2*\a*0.6 + 2*\b*0.8, 0)}]
\draw[line width = 1.4 mm, black!30, rotate=36.869897645844021]  (0, 0) arc (-180:-90:0.5*\a) -- (\c, -0.5*\a);
\draw[line width = 1.4 mm, black!30, rotate=36.869897645844021]  (\b - \c, -0.5*\a) -- (\b - 0.5*\a, -0.5*\a) arc (90:0:0.5*\a);
\draw[line width = 1.4 mm, black!30,rotate=-36.869897645844021] (0, 0) arc (0:-90:0.5*\a) -- (-\c, -0.5*\a);

\draw[line width = 1.4 mm, black!30, shift={(\a*0.6 + \b*0.8, -\a*0.8 + \b*0.6)}, rotate=-36.869897645844021] (0, 0) arc (180:90:0.5*\a) -- (\c, 0.5*\a);
\end{scope}

\draw[densely dashed] (-\c*0.8, \c*0.6) -- (-\c*0.8 - 0.5*\a*0.6, \c*0.6 - 0.5*\a*0.8);

\draw[densely dashed] (\c*0.8, \c*0.6) -- (\c*0.8 + 0.5*\a*0.6, \c*0.6 - 0.5*\a*0.8);
\draw[densely dashed] (\a*0.6 + \b*0.8 - \c*0.8, -\a*0.8 + \b*0.6 - \c*0.6) -- (0.5*\a*0.6 + \b*0.8 - \c*0.8, -0.5*\a*0.8 + \b*0.6 - \c*0.6);
\draw[densely dashed] (\a*0.6 + \b*0.8 + \c*0.8, -\a*0.8 + \b*0.6 - \c*0.6) -- (1.5*\a*0.6 + \b*0.8 + \c*0.8, -0.5*\a*0.8 + \b*0.6 - \c*0.6);

\draw[very thick, decoration={markings, mark=at position 0.28 with {\draw (-0.2, 0.07) -- (0, 0) -- (-0.2, -0.07);}, mark=at position 0.78 with {\draw (-0.2, 0.07) -- (0, 0) -- (-0.2, -0.07);}}, postaction={decorate}] (-\c*0.8, \c*0.6) -- (0, 0) -- (\c*0.8, \c*0.6);
\draw[very thick, decoration={markings, mark=at position 0.28 with {\draw (-0.2, 0.07) -- (0, 0) -- (-0.2, -0.07);}, mark=at position 0.78 with {\draw (-0.2, 0.07) -- (0, 0) -- (-0.2, -0.07);}}, postaction={decorate}] (\a*0.6 + \b*0.8 - \c*0.8, -\a*0.8 + \b*0.6 - \c*0.6) -- (\a*0.6 + \b*0.8, -\a*0.8 + \b*0.6) -- (\a*0.6 + \b*0.8 + \c*0.8, -\a*0.8 + \b*0.6 - \c*0.6);
\draw[rotate=36.869897645844021] (0, 0) arc (-180:-90:0.5*\a) -- (\b - 0.5*\a, -0.5*\a) arc (90:0:0.5*\a);
\draw[shift={(\a*0.6 + \b*0.8, -\a*0.8 + \b*0.6)}, rotate=-36.869897645844021] (0, 0) arc (180:90:0.5*\a) -- (\b - 0.5*\a, 0.5*\a) arc (-90:0:0.5*\a);

\draw[rotate=45, <->, thick] (0.32*\c, 0) arc (0:90:0.32*\c);
\draw[shift={ (\a*0.6 + \b*0.8, -\a*0.8 + \b*0.6)}, rotate=-45, <->, thick] (0.32*\c, 0) arc (0:-90:0.32*\c);

\draw[xscale=-1, rotate=36.869897645844021] (0, 0) arc (-180:-90:0.5*\a) -- (\b - 0.5*\a, -0.5*\a);

\draw (-\c*0.8 + \a*0.6*0.35, \c*0.6 + \a*0.8*0.35) node {$z_{1}$}; 
\draw (\c*0.8 - \a*0.6*0.35, \c*0.6+ \a*0.8*0.35) node {$y_{1}$}; 
\draw (\a*0.6 + \b*0.8-\c*0.8 + \a*0.6*0.35 , -\a*0.8 + \b*0.6 - \c*0.6 - \a*0.8*0.35) node {$z_{2}$};
\draw (\a*0.6 + \b*0.8+\c*0.8 - \a*0.6*0.35 , -\a*0.8 + \b*0.6 - \c*0.6 - \a*0.8*0.35) node {$y_{2}$};
\draw (0, -0.35) node {$x_{1}$};
\draw (\a*0.6 + \b*0.8, -\a*0.8 + \b*0.6+0.4) node {$x_{2}$}; 

\draw (-\b*0.8 + \a*0.1 -0.2, \b*0.6 - \a*0.7+0.1) node {$x$};

\begin{scope}[shift={(2*\a*0.6 + 2*\b*0.8, 0)}]

\draw[densely dashed] (-\c*0.8, \c*0.6) -- (-\c*0.8 -0.5*\a*0.6, \c*0.6 - 0.5*\a*0.8);
\draw[densely dashed] (\c*0.8, \c*0.6) -- (\c*0.8 + 0.5*\a*0.6, \c*0.6 - 0.5*\a*0.8);
\draw[densely dashed] (\a*0.6 + \b*0.8 - \c*0.8, -\a*0.8 + \b*0.6 - \c*0.6) -- (0.5*\a*0.6 + \b*0.8 - \c*0.8, -0.5*\a*0.8 + \b*0.6 - \c*0.6);

\draw[densely dashed] (\a*0.6 + \b*0.8 + \c*0.8, -\a*0.8 + \b*0.6 - \c*0.6) -- (1.5*\a*0.6 + \b*0.8 + \c*0.8, -0.5*\a*0.8 + \b*0.6 - \c*0.6);

\draw[very thick, decoration={markings, mark=at position 0.28 with {\draw (-0.2, 0.07) -- (0, 0) -- (-0.2, -0.07);}, mark=at position 0.78 with {\draw (-0.2, 0.07) -- (0, 0) -- (-0.2, -0.07);}}, postaction={decorate}] (-\c*0.8, \c*0.6) -- (0, 0) -- (\c*0.8, \c*0.6);
\draw[very thick, decoration={markings, mark=at position 0.28 with {\draw (-0.2, 0.07) -- (0, 0) -- (-0.2, -0.07);}, mark=at position 0.78 with {\draw (-0.2, 0.07) -- (0, 0) -- (-0.2, -0.07);}}, postaction={decorate}] (\a*0.6 + \b*0.8 - \c*0.8, -\a*0.8 + \b*0.6 - \c*0.6) -- (\a*0.6 + \b*0.8, -\a*0.8 + \b*0.6) -- (\a*0.6 + \b*0.8 + \c*0.8, -\a*0.8 + \b*0.6 - \c*0.6);
\draw[rotate=36.869897645844021] (0, 0) arc (-180:-90:0.5*\a) -- (\b - 0.5*\a, -0.5*\a) arc (90:0:0.5*\a);
\draw[shift={(\a*0.6 + \b*0.8, -\a*0.8 + \b*0.6)}, rotate=-36.869897645844021] (0, 0) arc (180:90:0.5*\a) -- (\b - 0.5*\a, 0.5*\a);

\draw[rotate=45, <->, thick] (0.32*\c, 0) arc (0:90:0.32*\c);
\draw[shift={ (\a*0.6 + \b*0.8, -\a*0.8 + \b*0.6)}, rotate=-45, <->, thick] (0.32*\c, 0) arc (0:-90:0.32*\c);

\draw (-\c*0.8 + \a*0.6*0.35, \c*0.6 + \a*0.8*0.35) node {$z_{3}$}; 
\draw (\c*0.8- \a*0.6*0.35, \c*0.6+ \a*0.8*0.35) node {$y_{3}$}; 
\draw (\a*0.6 + \b*0.8-\c*0.8 + \a*0.6*0.35 , -\a*0.8 + \b*0.6 - \c*0.6 - \a*0.8*0.35) node {$z_{4}$};
\draw (\a*0.6 + \b*0.8+\c*0.8 - \a*0.6*0.35 , -\a*0.8 + \b*0.6 - \c*0.6 - \a*0.8*0.35) node {$y_{4}$};

\draw (0, -0.37) node {$x_{3}$};
\draw (\a*0.6 + \b*0.8, -\a*0.8 + \b*0.6+0.4) node {$x_{4}$};

\draw[shift={(\b*0.8+\a*0.6,  \b*0.6 - \a*0.8 )}, rotate=180] (-\b*0.8 + \a*0.1 -0.2, \b*0.6 - \a*0.7+0.1) node {$y$};

\end{scope}

\end{tikzpicture}
\caption{Schematics for $[x, y]$ being $D$-marked with $([x_{i}, y_{i}])_{i=1}^{4}$, $([z_{i}, x_{i}])_{i=1}^{4}$.}
\label{fig:scheme2}
\end{figure}

The following observation is immediate.

\begin{lem}\label{lem:markedConcat}
Let $\{x_{i}\}_{i=0}^{N}$, $\{y_{i-1}, z_{i}\}_{i=1}^{N}$ be points in $X$. Suppose that for each $i=1, \ldots, N$, $[x_{i-1}, x_{i}]$ is fully $D$-marked with sequences of segments $(\gamma_{i, j})_{j=1}^{n_{i} - 1}$, $(\eta_{i, j})_{j=2}^{n_{i}}$, where $\gamma_{i, 1} = [x_{i-1}, y_{i-1}]$ and $\eta_{i, n_{i}} = [z_{i}, x_{i}]$. Suppose also that $(y_{i}, z_{i})_{x_{i}} \le D$ for $i=1, \ldots, N-1$. Then $[x_{0}, x_{N}]$ is fully $D$-marked with \[\begin{aligned}
\left(\gamma_{1, 1}, \ldots, \gamma_{1, n_{1} - 1},\,\, \gamma_{2, 1}, \ldots, \gamma_{2, n_{2} - 1},\,\, \ldots, \gamma_{N, n_{N} - 1}\right), \\
\left(\eta_{1, 2}, \ldots, \eta_{1, n_{1}},\quad\,\, \eta_{2, 2}, \ldots, \eta_{2, n_{2}},\quad\,\,\ldots, \eta_{N, n_{N}}\quad\right).
\end{aligned}
\]
\end{lem}

Our aim is to prove that if $[x, y]$ is $D$-marked with sufficiently thick and long segments, then $[x, y]$ is witnessed by those segments. In order to prove this, we need the following consequences of Rafi's fellow-travelling and thin triangle theorems \cite{rafi2014hyperbolicity}.

\begin{fact}[{\cite[Lemma 3.9]{choi2021clt}}]\label{lem:concat}
For each $D, \epsilon > 0$, there exist $E, L > D$ that satisfy the following property. Let $x_{1}$ be an $\epsilon$-thick point and $\{x_{i}\}_{i=2}^{3}, \{y_{i}\}_{i=1}^{5}$ be points in $X$ such that: \begin{enumerate}
\item $[y_{1}, y_{2}]$, $[y_{3}, y_{4}]$, $[y_{4}, y_{5}]$ are $\epsilon$-thick and longer than $L$,
\item $(y_{3}, y_{5})_{y_{4}} \le D$,
\item $[x_{1}, x_{2}]$ is $D$-witnessed by $([y_{1}, y_{2}], [y_{3}, y_{4}])$, and 
\item $[x_{2}, x_{3}]$ is $E$-witnessed by $[y_{4}, y_{5}]$,
\end{enumerate}
then $[x_{1}, x_{3}]$ is $E$-witnessed by $[y_{1}, y_{2}]$.
\end{fact}

\begin{fact}[{\cite[Lemma 3.10]{choi2021clt}}]\label{lem:concatUlt}
For each $E$, $\epsilon > 0$, there exists $F, L> E$ that satisfies the following property. Let $\{x_{i}\}_{i=1}^{3}, \{y_{i}\}_{i=1}^{3}$ be points in $X$ such that: \begin{enumerate}
\item $[y_{1}, y_{2}]$, $[y_{2}, y_{3}]$ are $\epsilon$-thick and longer than $L$,
\item $(y_{1}, y_{3})_{y_{2}} \le E$,
\item $[x_{1}, x_{2}]$ is $E$-witnessed by $[y_{1}, y_{2}]$, and
\item $[x_{2}, x_{3}]$ is $E$-witnessed by $[y_{2}, y_{3}]$,
\end{enumerate}
then $[x_{1}, x_{3}]$ is $F$-witnessed by $[y_{1}, y_{2}]$, and it is also $F$-witnessed by $[y_{2}, y_{3}]$. In particular, $|(x_{1}, x_{3})_{x_{2}} - d(x_{2}, y_{2})| < F$.
\end{fact}

\begin{figure}[h]
\begin{tikzpicture}[scale=0.85]
\begin{scope}[shift={(-0.3, -3)}]

\draw[line width = 1.4 mm, black!30]  (0, 0.3) -- (1.5, 0.3);
\draw[line width = 1.4 mm, black!30]  (2, 0.3) -- (3.5 - 0.096856471680698 - 3*0.15, 0.3) arc (90:90-30:0.15*4.162277660168379);
\draw[line width = 1.4 mm, black!30, shift={(5, 1.3)}, rotate=36.869897645844021]  (0, 0) -- (-1.5+ 0.096856471680698+3*0.15, 0) arc (90:90+30:0.15*4.162277660168379);

\draw[thick, <->, shift={(3.5, 0.8)}, rotate=-7] (-0.35, 0) arc (180:52:0.35);
\draw[very thick, decoration={markings, mark=at position 0.57 with {\draw (-0.2, 0.07) -- (0, 0) -- (-0.2, -0.07);}}, postaction={decorate}] (0, 0) -- (1.5, 0);
\draw[very thick, decoration={markings, mark=at position 0.28 with {\draw (-0.2, 0.07) -- (0, 0) -- (-0.2, -0.07);}, mark=at position 0.78 with {\draw (-0.2, 0.07) -- (0, 0) -- (-0.2, -0.07);}}, postaction={decorate}] (2, 0.8) -- (3.5,0.8) -- (4.7, 1.7);
\draw (-0.7, 0.5) arc (-180:-90:0.2) -- (3.5 - 0.096856471680698 - 3*0.15, 0.3) arc (90:90-71.565051177077989:0.15*4.162277660168379)-- (3.5 + 0.6 - 0.25*0.948683298050514, 0.8 - 0.6*3 - 0.25*0.316227766016838) arc (18.434948822922011-180:18.434948822922011:0.25) -- (3.5 - 0.096856471680698+ 0.948683298050514*0.15 + 2*0.237170824512628, 0.3 - 0.15*2.846049894151541 + 2*0.079056941504209) arc (270 - 71.565051177077989:270 - 2*71.565051177077989:0.15*4.162277660168379) -- (5.7, 1.825) arc (36.869897645844021-90:36.869897645844021:0.2);

\draw[dotted] (-0.7, 0.5) -- (3.6 - 0.4, 0.5) arc (-90:-90+36.869897645844021:0.4*3) -- (5.74, 2.085);
\draw[densely dashed] (0, 0) -- (0, 0.3);
\draw[densely dashed] (1.5, 0) -- (1.5, 0.3);
\draw[densely dashed] (2, 0.8) -- (2, 0.3);
\draw[densely dashed] (3.5, 0.8) -- (3.25, 0.23);
\draw[densely dashed] (3.5, 0.8) -- (4.05, 0.52);
\draw[densely dashed] (4.7, 1.7) -- (5, 1.3);

\fill (-0.7, 0.5) circle (0.065);
\fill  (3.5 + 0.6 +0.25*0.316227766016838, 0.8 - 0.6*3 - 0.25*0.948683298050514) circle (0.065);
\fill  (5.74, 2.085) circle (0.065);

\draw (6.15, 0.4) node {\large $\Rightarrow$};

\draw (-1, 0.6) node {$x_{1}$};
\draw (4.37, -1.51) node {$x_{2}$};
\draw (6.11, 2.23) node {$x_{3}$};

\draw (0.1, -0.3) node {$y_{1}$};
\draw (1.44, -0.3) node {$y_{2}$};
\draw (2.08, 1.08) node {$y_{3}$};
\draw (3.65, 0.5) node {$y_{4}$};
\draw (4.47, 1.9) node {$y_{5}$};

\end{scope}

\begin{scope}[shift={(7.3, -3)}]

\draw[line width = 1.4 mm, black!30]  (0, 0.5) -- (1.5, 0.5);

\draw[very thick, decoration={markings, mark=at position 0.57 with {\draw (-0.2, 0.07) -- (0, 0) -- (-0.2, -0.07);}}, postaction={decorate}] (0, 0) -- (1.5, 0);
\draw[very thick, decoration={markings, mark=at position 0.28 with {\draw (-0.2, 0.07) -- (0, 0) -- (-0.2, -0.07);}, mark=at position 0.78 with {\draw (-0.2, 0.07) -- (0, 0) -- (-0.2, -0.07);}}, postaction={decorate}] (2, 0.8) -- (3.5,0.8) -- (4.7, 1.7);
\draw[dotted] (-0.7, 0.5) arc (-180:-90:0.2) -- (3.5 - 0.096856471680698 - 3*0.15, 0.3) arc (90:90-71.565051177077989:0.15*4.162277660168379)-- (3.5 + 0.6 - 0.25*0.948683298050514, 0.8 - 0.6*3 - 0.25*0.316227766016838) arc (18.434948822922011-180:18.434948822922011:0.25) -- (3.5 - 0.096856471680698+ 0.948683298050514*0.15 + 2*0.237170824512628, 0.3 - 0.15*2.846049894151541 + 2*0.079056941504209) arc (270 - 71.565051177077989:270 - 2*71.565051177077989:0.15*4.162277660168379) -- (5.7, 1.825) arc (36.869897645844021-90:36.869897645844021:0.2);
\draw (-0.7, 0.5) -- (3.6 - 0.3, 0.5) arc (-90:-90+36.869897645844021:0.3*3) -- (5.74, 2.085);
\draw[densely dashed] (0, 0) -- (0, 0.5);
\draw[densely dashed] (1.5, 0) -- (1.5, 0.5);

\fill (-0.7, 0.5) circle (0.065);
\fill  (3.5 + 0.6 +0.25*0.316227766016838, 0.8 - 0.6*3 - 0.25*0.948683298050514) circle (0.065);
\fill  (5.74, 2.085) circle (0.065);

\end{scope}

\begin{scope}[shift={(-0.5, -7.4)}]

\draw[line width = 1.4 mm, black!30]  (2, 0.3) -- (3.5 - 0.096856471680698 - 3*0.15, 0.3) arc (90:90-30:0.15*4.162277660168379);
\draw[line width = 1.4 mm, black!30, shift={(5, 1.3)}, rotate=36.869897645844021]  (0, 0) -- (-1.5+ 0.096856471680698+3*0.15, 0) arc (90:90+30:0.15*4.162277660168379);

\draw[thick, <->, shift={(3.5, 0.8)}, rotate=-7] (-0.35, 0) arc (180:52:0.35);
\draw[very thick, decoration={markings, mark=at position 0.28 with {\draw (-0.2, 0.07) -- (0, 0) -- (-0.2, -0.07);}, mark=at position 0.78 with {\draw (-0.2, 0.07) -- (0, 0) -- (-0.2, -0.07);}}, postaction={decorate}] (2, 0.8) -- (3.5,0.8) -- (4.7, 1.7);
\draw (0.8, 0.5) arc (-180:-90:0.2) -- (3.5 - 0.096856471680698 - 3*0.15, 0.3) arc (90:90-71.565051177077989:0.15*4.162277660168379)-- (3.5 + 0.6 - 0.25*0.948683298050514, 0.8 - 0.6*3 - 0.25*0.316227766016838) arc (18.434948822922011-180:18.434948822922011:0.25) -- (3.5 - 0.096856471680698+ 0.948683298050514*0.15 + 2*0.237170824512628, 0.3 - 0.15*2.846049894151541 + 2*0.079056941504209) arc (270 - 71.565051177077989:270 - 2*71.565051177077989:0.15*4.162277660168379) -- (5.7, 1.825) arc (36.869897645844021-90:36.869897645844021:0.2);

\draw[dotted] (0.8, 0.5) -- (3.6 - 0.4, 0.5) arc (-90:-90+36.869897645844021:0.4*3) -- (5.74, 2.085);

\draw[densely dashed] (2, 0.8) -- (2, 0.3);
\draw[densely dashed] (3.5, 0.8) -- (3.25, 0.23);
\draw[densely dashed] (3.5, 0.8) -- (4.05, 0.52);
\draw[densely dashed] (4.7, 1.7) -- (5, 1.3);

\fill (0.8, 0.5) circle (0.065);
\fill  (3.5 + 0.6 +0.25*0.316227766016838, 0.8 - 0.6*3 - 0.25*0.948683298050514) circle (0.065);
\fill  (5.74, 2.085) circle (0.065);

\draw (6.2, 0.4) node {\large $\Rightarrow$};

\draw (0.45, 0.6) node {$x_{1}$};
\draw (4.28, -1.52) node {$x_{2}$};
\draw (6.2, 2.2) node {$x_{3}$};

\draw (2, 1.07) node {$y_{1}$};
\draw (3.66, 0.48) node {$y_{2}$};
\draw (4.8, 1.9) node {$y_{3}$};

\end{scope}

\begin{scope}[shift={(6.1, -7.4)}]

\draw[line width = 1.4 mm, black!30]  (2, 0.5) -- (3.33, 0.5);
\draw[line width = 1.4 mm, black!30, shift={(4.88, 1.46)}, rotate=36.869897645844021]  (0, 0) -- (-1.33, 0);

\draw[very thick, decoration={markings, mark=at position 0.28 with {\draw (-0.2, 0.07) -- (0, 0) -- (-0.2, -0.07);}, mark=at position 0.78 with {\draw (-0.2, 0.07) -- (0, 0) -- (-0.2, -0.07);}}, postaction={decorate}] (2, 0.8) -- (3.5,0.8) -- (4.7, 1.7);
\draw[dotted] (0.8, 0.5) arc (-180:-90:0.2) -- (3.5 - 0.096856471680698 - 3*0.15, 0.3) arc (90:90-71.565051177077989:0.15*4.162277660168379)-- (3.5 + 0.6 - 0.25*0.948683298050514, 0.8 - 0.6*3 - 0.25*0.316227766016838) arc (18.434948822922011-180:18.434948822922011:0.25) -- (3.5 - 0.096856471680698+ 0.948683298050514*0.15 + 2*0.237170824512628, 0.3 - 0.15*2.846049894151541 + 2*0.079056941504209) arc (270 - 71.565051177077989:270 - 2*71.565051177077989:0.15*4.162277660168379) -- (5.7, 1.825) arc (36.869897645844021-90:36.869897645844021:0.2);
\draw (0.8, 0.5) -- (3.6 - 0.3, 0.5) arc (-90:-90+36.869897645844021:0.3*3) -- (5.74, 2.085);
\draw[densely dashed] (2, 0.8) -- (2, 0.5);
\draw[densely dashed] (3.5, 0.8) -- (3.33, 0.5);
\draw[densely dashed] (3.5, 0.8) -- (3.82, 0.65);
\draw[densely dashed] (4.7, 1.7) -- (4.88, 1.46);

\fill (0.8, 0.5) circle (0.065);
\fill  (3.5 + 0.6 +0.25*0.316227766016838, 0.8 - 0.6*3 - 0.25*0.948683298050514) circle (0.065);
\fill  (5.74, 2.085) circle (0.065);
\end{scope}

\end{tikzpicture}
\caption{Schematics for Fact \ref{lem:concat} and \ref{lem:concatUlt}.}
\label{fig:scheme1}
\end{figure}
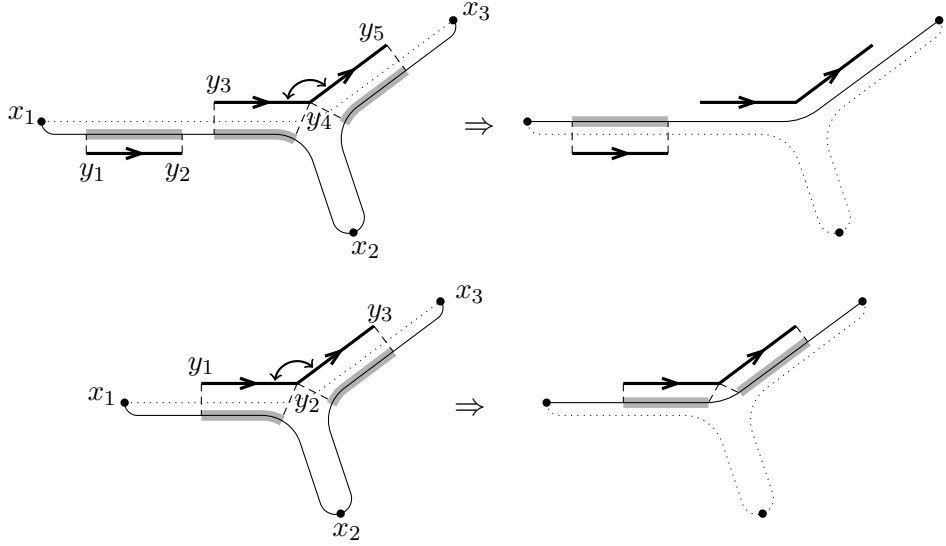

Combining Fact \ref{lem:concat} and \ref{lem:concatUlt} yields the following.

\begin{cor}\label{cor:induction}
Let $D, \epsilon > 0$ and \begin{itemize}
\item $E = E(\epsilon, D)$, $L_{1} = L(\epsilon, D)$ as in Fact \ref{lem:concat}, and 
\item $F = F(\epsilon, E)$, $L_{2} = L(\epsilon, E)$ as in Fact \ref{lem:concatUlt}.
\end{itemize}
Suppose that $x, y, \{x_{i}, y_{i}, z_{i}\}_{i=1}^{N}$ are points in $X$ such that: \begin{enumerate}
\item $[x, y]$ is $D$-marked with $([x_{i}, y_{i}])_{i=1}^{N}$, $([z_{i}, x_{i}])_{i=1}^{N}$, and
\item $[x_{i}, y_{i}]$, $[z_{i}, x_{i}]$ are $\epsilon$-thick and longer than $\max(L_{1}, L_{2})$ for $i=1, \ldots, N$.
\end{enumerate}
Then we have that: \begin{enumerate}
\item $[x, y]$ is $F$-witnessed by $[x_{i}, y_{i}]$ for each $i=1, \ldots, N$, 
\item $[x, y]$ is $F$-witnessed by $[z_{i}, y_{i}]$ for each $i=1, \ldots, N$,  and
\item $(x, y)_{x_{i}}$, $(x, y)_{y_{i}}$, $(x, y)_{z_{i}}$ are smaller than $F$ for each $i=1, \ldots, N$.
\end{enumerate}
\end{cor}

\begin{proof}
It is assumed that $[x_{N}, y]$ is $E$-witnessed by $[x_{N}, y_{N}]$. Moreover, by assumption, $[x_{k-1}, x_{k}]$ is $D$-witnessed by $([x_{k-1}, y_{k-1}], [z_{k}, x_{k}])$ where $[x_{k-1}, y_{k-1}]$, $[z_{k}, x_{k}]$ are $\epsilon$-thick and longer than $L_{1}$. Note also that $(y_{k}, z_{k})_{x_{k}} \le D$. Hence, if $[x_{k}, y]$ is $E$-witnessed by $[x_{k}, y_{k}]$, then $[x_{k-1}, y]$ is $E$-witnessed by $[x_{k-1}, y_{k-1}]$ by Fact \ref{lem:concat}. Thus, inductively, we deduce that $[x_{i}, y]$ is $E$-witnessed by $[x_{i}, y_{i}]$ for each $i$, Similarly, $[x, x_{i}]$ is $E$-witnessed by $[z_{i}, x_{i}]$. Now Fact \ref{lem:concatUlt} asserts that $[x, y]$ is $F$-witnessed by $[x_{i}, y_{i}]$ and $[z_{i}, y_{i}]$, which also implies the second item.
\end{proof}

We finally need two facts that guarantee witnessing by a pair of segments.

\begin{fact}[{\cite[Lemma 3.6]{choi2021clt}}]\label{lem:farSegment}
For each $C, \epsilon > 0$, there exists $D >C$ that satisfies the following. If $x, y, z, x' \in X$ satisfy that\begin{enumerate}
\item $[x, y]$, $[z, x']$ are $\epsilon$-thick;
\item $(x, z)_{y}$, $(y, x')_{z} < C$, and
\item $d(y, z) \ge \max\{d(x, y), d(z, x'), 3D\}$,
\end{enumerate}
then $[x, x']$ is $D$-witnessed by $([x, y], [z, x'])$.
\end{fact}

\begin{fact}[{\cite[Lemma 3.7]{choi2021clt}}]\label{lem:1segment}
For each $C, \epsilon> 0$, there exists $D>C$ that satisfies the following condition. If $x$, $y$, $x'$, $z \in X$ satisfy that \begin{enumerate}
\item $[x, y]$, $[x', z]$ are $\epsilon$-thick and
\item $(x, z)_{y}$, $(x, x')_{z} < C$,
\end{enumerate}
then $[x, x']$ is $D$-witnessed by $([x, y], [z, x'])$.
\end{fact}

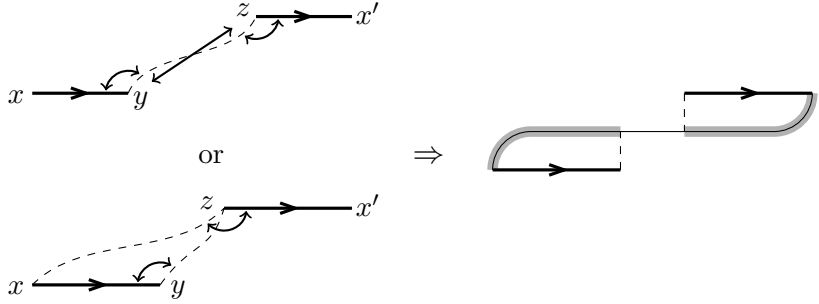
\begin{figure}[h]
\centering
\begin{tikzpicture}[scale=0.85]

\draw[very thick, decoration={markings, mark=at position 0.56 with {\draw (-0.2, 0.07) -- (0, 0) -- (-0.2, -0.07);}}, postaction={decorate}] (0, 0) -- (1.5, 0);
\draw[very thick, decoration={markings, mark=at position 0.556with {\draw (-0.2, 0.07) -- (0, 0) -- (-0.2, -0.07);}}, postaction={decorate}] (3.5, 1.2) -- (5, 1.2);
\draw[dashed] (1.5, 0) .. controls (1.8, 0.7) and (3.2, 0.5) .. (3.5, 1.2);
\draw[thick, <->, shift={(1.5, 0)}, rotate=-7] (-0.35, 0) arc (180:67:0.35);
\draw[thick, <->, shift={(3.5, 1.2)}, rotate=-7] (0.35, 0) arc (0:-113:0.35);
\draw[thick, <->] (1.87, 0.18) -- (3.13, 1.02);

\draw (-0.25, -0.05) node {$x$};
\draw (1.7, -0.11) node {$y$};
\draw (3.3, 1.31) node {$z$};
\draw (5.25, 1.25) node {$x'$};

\draw (2.8, -1) node {or};

\draw(6.2, -1) node {\large $\Rightarrow$};

\begin{scope}[shift={(0, -3)}]

\draw[very thick, decoration={markings, mark=at position 0.56 with {\draw (-0.2, 0.07) -- (0, 0) -- (-0.2, -0.07);}}, postaction={decorate}] (0, 0) -- (2, 0);
\draw[very thick, decoration={markings, mark=at position 0.556with {\draw (-0.2, 0.07) -- (0, 0) -- (-0.2, -0.07);}}, postaction={decorate}] (3, 1.2) -- (5, 1.2);
\draw[dashed] (2, 0) .. controls (2.5, 0.7) and (2.8, 0.5) .. (3, 1.2);
\draw[dashed] (0, 0) .. controls (0.8, 0.8) and (2.2, 0.4) .. (3, 1.2);
\draw[thick, <->, shift={(2, 0)}, rotate=-7] (-0.35, 0) arc (180:67:0.35);
\draw[thick, <->, shift={(3, 1.2)}, rotate=-7] (0.35, 0) arc (0:-130:0.35);

\draw (-0.25, -0.05) node {$x$};
\draw (2.28, -0.08) node {$y$};
\draw (2.73, 1.28) node {$z$};
\draw (5.25, 1.25) node {$x'$};

\end{scope}

\begin{scope}[shift={(7.2, -1.2)}]
\draw[line width = 1.4 mm, black!30]  (0, 0) arc (180:90:0.6) -- (2, 0.6);
\draw[line width = 1.4 mm, black!30]  (3, 0.6) --(4.4, 0.6) arc (-90:0:0.6);
\draw[very thick, decoration={markings, mark=at position 0.56 with {\draw (-0.2, 0.07) -- (0, 0) -- (-0.2, -0.07);}}, postaction={decorate}] (0, 0) -- (2, 0);
\draw[very thick, decoration={markings, mark=at position 0.56 with {\draw (-0.2, 0.07) -- (0, 0) -- (-0.2, -0.07);}}, postaction={decorate}] (3, 1.2) -- (5, 1.2);
\draw (0, 0) arc (180:90:0.6) -- (4.4, 0.6) arc (-90:0:0.6);
\draw[dashed] (2, 0) -- (2, 0.6);
\draw[dashed] (3, 1.2) -- (3, 0.6);

\end{scope}

\end{tikzpicture}
\caption{Schematics for Fact \ref{lem:farSegment}, \ref{lem:1segment}.}
\label{fig:scheme1}
\end{figure}

\emph{From now on, we permanently fix a base point $o \in X$ and  a finitely generated subgroup $G$ of $\Isom(X)$ that contains independent loxodromics $a$, $b$.} For $S \subset G$ and $i \in \Z$, we employ the notation $S^{(i)} := \{s^{i} : s \in S\}$.

Let us introduce the notion of Schottky sets that originated from \cite{gouezel2021exp}.

\begin{definition}[Schottky set]\label{dfn:Schottky}
Let $K, K', \epsilon> 0$. A finite set $S$ of isometries of $X$ is said to be \emph{$(K, K', \epsilon)$-Schottky} if the following hold: \begin{enumerate}
\item for all $x, y \in X$, $|\{s \in S : (x, s^{i}y)_{o} \ge K\,\,\textrm{for some}\,\,i>0\}| \le2$;
\item for all $x, y \in X$, $|\{s \in S : (x, s^{i}y)_{o} \ge K\,\,\textrm{for some}\,\,i<0\}| \le2$;
\item for all $s \in S$ and $i \neq 0$, $0.9995 iK' < d(o, s^{i} o)\le iK'$;
\item for all $s \in S$ and $i \in \Z$, the geodesic $[o, s^{i}o]$ is $\epsilon$-thick;
\item for all $x \in X$, $|\{s \in S : (x, s^{i} o)_{o} \ge K\,\,\textrm{for some}\,\,i>0\}| \le1$;
\item for all $x \in X$, $|\{s \in S : (x, s^{i} o)_{o} \ge K\,\,\textrm{for some}\,\,i<0\}| \le1$, and
\item for all $s_{1}, s_{2} \in S$ and $i, j > 0$, $(s_{1}^{i}o, s_{2}^{-j} o)_{o} < K$.
\end{enumerate}
\end{definition}

\begin{remark}
Gou{\"e}zel's original definition of Schottky set in \cite{gouezel2021exp} and Choi's definition in \cite{choi2021clt} do not require (5), (6) and (7); however, as remarked there, Schottky sets constructed in \cite{choi2021clt} automatically satisfy (5), (6) and (7).
\end{remark}

We now claim the existence of certain Schottky sets.

\begin{prop}[{cf. \cite[Proposition 4.2]{choi2021clt}}]\label{prop:Schottky}
There exists $\epsilon > 0$ such that the following holds. For any $n, n'>0$, there exist $K(n)>0$ and $K'(n, n')> n'$ such that $G$ has a $(K, K', \epsilon)$-Schottky subset $S$ with cardinality at least $n$.
\end{prop}

\begin{proof}
The proof of \cite[Proposition 4.2]{choi2021clt} implies the following: 

\begin{claim}\label{claim:prop4.2}
There exists $\epsilon, F, N_{0} >0$ such that the following holds for all $N>N_{0}$: \begin{enumerate}[label=(\alph*)]
\item  for any $0\le m \le n$ and $\phi_{i} \in \{a, b, a^{-1}, b^{-1}\}$ such that $\phi_{i} \notin \phi_{i+1}^{-1}$, $[o, \phi_{1}^{2N} \cdots \phi_{n}^{2N} o]$ is $\epsilon$-thick and $(o, \phi_{1}^{2N} \cdots \phi_{n}^{2N} o)_{\phi_{1}^{2N} \cdots \phi_{m}^{2N} o} \le F$;
\item for all $n$ and $k$, the set \[
S_{n, k, N}:= \{(\phi_{1}^{2N} \cdots \phi_{n}^{2N})^{2k} : \phi_{i} \in \{a, b\}\}
\] satisfies Property (1), (2), (4), (5) and (6) of Schottky sets for \[
K(n, N) := \max\{ d(o, \phi_{1}^{2N} \cdots \phi_{n}^{2N} o) : \phi_{i} \in \{a, b\}\}.
\] 
\end{enumerate}
\end{claim}

The proof of this claim is given in Appendix \ref{section:remark}. Assuming this, we now take large enough $N$ such that $d(o, a^{2N} o), d(o, b^{2N} o) > 10000 F$ and fix $K(n): = K(n, N)$. Let $\mathcal{S}$ be the collection of concatenations of $n$ copies of $a^{2N}$ and $n$ copies of $b^{2N}$. For $s \in \mathcal{S}$ and $k$, the property (a) above implies\[
0 \le [2nkd(o, a^{2N} o) +2nkd(o, b^{2N} o)]- d(o, s^{k} o)\le 8nkF \le \frac{1}{2500} \cdot 2nk[d(o, a^{2N} o), d(o, b^{2N} o)].
\]
We finally fix $k$ such that $K'(n, n') := 2nk[d(o, a^{2N} o) + d(o, b^{2N} o)]$ is larger than $n'$. Then $\mathcal{S}^{(k)}$ is $(K, K', \epsilon)$-Schottky and has cardinality $\binom{2n}{n} \ge n$.
\end{proof}

\emph{From now on, we permanently fix the choice $\epsilon >0$ from Proposition \ref{prop:Schottky}}. Now for each $C>0$, we fix: \begin{itemize}
\item $D = D(C, \epsilon)$ that works in Fact \ref{lem:farSegment} and \ref{lem:1segment},
\item $E = E(D, \epsilon), L_{1} = L(D, \epsilon)$ as in Fact \ref{lem:concat},
\item $F = F(E, \epsilon), L_{2} = L(D, \epsilon)$ as in Fact \ref{lem:concatUlt}
\end{itemize}
and $L = \max(L_{1}, L_{2})$. Note that $D$, $F$ and $L$ ultimately depend on the values of $C$ and $\epsilon$; we will write them as $D(C, \epsilon)$, $F(C, \epsilon)$ and $L(C, \epsilon)$.

\begin{lem}\label{lem:almostInj}
Let $K>0$, $K' > 2L(K, \epsilon) + 5000F(K, \epsilon)$ and $S_{1}$ be a $(K, K', \epsilon)$-Schottky set. Then $S_{1}$ and $S_{1}^{(-1)}$ are disjoint. 

Moreover, if a nonempty sequence $(s_{i})_{i=1}^{N}$ of elements of $S_{1} \cup S_{1}^{(-1)}$ satisfy $s_{i} \neq s_{i+1}^{-1}$ for $i = 1, \ldots, N-1$, then $s_{1} \cdots s_{N} \neq id$.
\end{lem}

\begin{proof}
For each $s \in S_{1}^{(\pm 1)}$, we claim that \[
\{s' \in S_{1} \cup S_{1}^{(-1)} : (so, s'o)_{o} > K\} = \{s\}.
\] Indeed, $(so, so)_{o} = d(o, so) > K$ and Property (5), (6), (7) of Schottky sets imply \[
\{s' \in S_{1} \cup S_{1}^{(-1)} : (so, s'o)_{o} > K\} = \{s' \in S_{1}^{(\pm 1)} : (so, s'o)_{o} > K\} = \{s\}.
\] In particular, $s' \in S_{1}^{(\mp 1)}$ cannot belong to this set; this settles the first claim.

We now let $x_{i} = s_{1} \cdots s_{i} o$ for $i=0, \ldots, N$. By the above claim, we realize that \[
(x_{i}, x_{i+2})_{x_{i}} = (s_{i+1}^{-1} o, s_{i+2}o)_{o}< K < D(K, \epsilon)
\] for $i=0, \ldots, N-2$. Then Fact \ref{lem:1segment} implies that $[x_{i}, x_{i+2}]$ is $D(K, \epsilon)$-witnessed by $([x_{i}, x_{i+1}], [x_{i+1}, x_{2}])$. Note also that $[x_{i}, x_{i+1}]$ is trivially $D(K, \epsilon)$-witnessed by $([x_{i}, x_{i+1}], [x_{i+1}, x_{i+1}])$, and by $([x_{i}, x_{i}], [x_{i}, x_{i+1}])$. Combining these observations, we deduce that $[x_{0}, x_{N}]$ is $D(K, \epsilon)$-marked with $([x_{2i-1}, x_{2i}])_{i=1}^{\lfloor N/2 \rfloor}$, $([x_{2i-2}, x_{2i-1}])_{i=1}^{\lfloor N/2 \rfloor}$. Since $[x_{i}, x_{i+1}]$ are $\epsilon$-thick and longer than $0.999K' > L(K, \epsilon)$,  Corollary \ref{cor:induction} implies that $(x_{0}, x_{i+1})_{x_{i}} < F(K, \epsilon)$ for each $i$ and \[
d(x_{0}, x_{N}) \ge \sum_{i=1}^{N} d(x_{i-1}, x_{i}) - 2(N-1) F(K, \epsilon) \ge (0.999K' - F(K, \epsilon))N \ge 0.9K'N.
\] Hence, $s_{1} \cdots s_{N} \neq id$.
\end{proof}

\begin{cor}\label{cor:almostInj}
Let $S_{1}$ be as in Lemma \ref{lem:almostInj}. Then the correspondence $a \mapsto a^{2}$ from $S_{1} \cup S_{1}^{(-1)}$ to $S_{1}^{(2)} \cup S_{1}^{(-2)}$ is 1-1.
\end{cor}

\begin{definition}\label{dfn:nice}
We say that a finite set $S$ is \emph{nicely populated} by $S_{0}$ if $S_{0} \subseteq S$ and $\#S_{0} \ge 0.99 \cdot \#S+ 400$.
\end{definition}

\begin{lem}\label{lem:Schottky}
Given a finite set $S' \subseteq G$, there exist a $(K, K', \epsilon)$-Schottky subset $S_{1}$ of $G$ such that $K' > 2L(K, \epsilon)+5000F(K, \epsilon)$, and a finite symmetric generating set $S$ of $G$ such that $S' \cup \{e\} \subseteq S$ and $S$ is nicely populated by $S_{1}^{(2)} \cup S_{1}^{(-2)}$. 
\end{lem}

\begin{proof}
We first enlarge $S'$ into a finite symmetric generating set $S''$ containing $e$. Let $n = 100\#S'' + 40000$ and take $K = K(n) > 0$ from Proposition \ref{prop:Schottky}. We then take $F = F(K, \epsilon)$ and $L = L(K, \epsilon)$ as described before. Using Proposition \ref{prop:Schottky}, we take $K'> n'=2L+5000F$ and a $(K, K', \epsilon)$-Schottky subset $S_{1}$ of $G$ with cardinality at least $n$. Thanks to Corollary \ref{cor:almostInj}, we also have $\#\left(S_{1}^{(2)} \cup S_{1}^{(-2)}\right)= 2 \cdot \#S_{1} \ge n$. Hence, the union of $S''$, $S_{1}^{(2)}$, and $S_{1}^{(-2)}$ satisfies the desired property.
\end{proof}

\emph{From now on, we fix constants $K>0$, $K' > 2L(K, \epsilon) + 5000F(K, \epsilon)$, a $(K,K', \epsilon)$-Schottky set $S_{1}$ with $K' > 2L(K, \epsilon)+5000F(K, \epsilon)$, and a finite symmetric generating set $S\ni e$ of $G$ that is nicely populated by $S_{1}^{(2)} \cup S_{1}^{(-2)}$.} For $g \in G$ and $s \in S_{1} \cup S_{1}^{(2)}$, we call $[go, gso]$ a \emph{Schottky segment}. One should keep in mind that Schottky segments are $\epsilon$-thick and longer than $0.999K'$. Finally, we will denote $S_{1} \cup S_{1}^{(-1)}$ by $S_{0}$.

%
%

\section{Pivoting in a random walk} \label{section:pivot}

We will make use of the random walk on $G$ generated by the uniform measure $\mu_{S}$ on $S$ that is constructed as follows. We consider the \emph{step space} $(G^{\Z}, \mu_{S}^{\Z})$, the product space of $G$ equipped with the product measure of $\mu_{S}$. Each step path $(g_{n})$ induces a sample path $(\w_{n})$ by \[
\w_{n} = \left\{ \begin{array}{cc} g_{1} \cdots g_{n} & n > 0 \\ e & n=0 \\ g_{0}^{-1} \cdots g_{n+1}^{-1} & n < 0, \end{array}\right.
\]
which constitutes a random walk with transition probability $\mu_{S}$.

Given $\mathbf{g} = (g_{1}, \ldots, g_{n}) \in G^{n}$, we define\[
\Theta(\mathbf{g})= \{ \vartheta(1) < \ldots < \vartheta(N)\} := \left\{1 \le i \le n/2 : g_{2i-1}, g_{2i} \in S_{1}^{(2)} \cup S_{1}^{(-2)}\right\}.
\] 
In other words, $\Theta(\mathbf{g})$ is the collection of steps that are chosen from $S_{1}^{(2)} \cup S_{1}^{(-2)}$. This set can well be empty, although such a situation happens with small probability. We now pick pivotal times from $\Theta(\mathbf{g})$.

For each $1 \le i \le N$, let $a_{i}, b_{i}$ be the elements of $S_{0}=S_{1} \cup S_{1}^{-1}$ such that \[
a_{i}^{2} = g_{2\vartheta(i) - 1}, \quad b_{i}^{2} = g_{2\vartheta(i)}. 
\]
Such $a_{i}$, $b_{i}$ are uniquely determined thanks to Corollary \ref{cor:almostInj}.

We also define $w_{i} := g_{2\vartheta(i) + 1} \cdots g_{2 \vartheta(i+1) - 2}$ for $1 \le i \le N-1$ and $w_{0} := g_{1} \cdots g_{2\vartheta(1) - 2}$, $w_{N} := g_{2\vartheta(N) +1} \cdots g_{n}$. It is clear that \[
\w_{n} = g_{1} g_{2} \cdots g_{n} = w_{0}a_{1}^{2} b_{1}^{2} w_{1} \cdots a_{N}^{2} b_{N}^{2} w_{N}.
\]

\begin{remark}
It will be convenient to allow expression $\w_{2\vartheta(N+1) - 2}$ and interpret it as $\w_{n}$, even though $\vartheta(N+1)$ does not exist (there is no reason to not define $\vartheta(N+1) := (n+2)/2$ if one hopes). This way, we can say: \[
\begin{aligned}
\w_{2\vartheta(k) - 2} &:= w_{0} a_{1}^{2}b_{1}^{2} \cdots w_{k-1} & (k = 1, \ldots, N+1), \\
\w_{2\vartheta(k) - 1} &:= w_{0} a_{1}^{2}b_{1}^{2} \cdots w_{k-1} a_{k}^{2} & (k = 1, \ldots, N), \\
\w_{2\vartheta(k)} &:= w_{0} a_{1}^{2}b_{1}^{2} \cdots w_{k-1} a_{k}^{2}b_{k}^{2} & (k = 1, \ldots, N).
\end{aligned}
\]
\end{remark}

\begin{figure}
\centering
\begin{tikzpicture}[scale=1.35]
\draw[thick, dashed] (-1.2, 1.2) -- (0, 0);
\draw[very thick] (0, 0) -- (0.84, 0.245) -- (1.085, 1.085) -- (2.2, 1.9);
\draw[thick, dashed, shift={(3.4, 0)}] (-1.2, 1.9) -- (0, 0);
\draw[very thick, shift={(3.4, 0)}] (0, 0) -- (0.84, 0.245) -- (1.085, 1.085) -- (2.2, 1.9);
\draw[thick, dashed, shift={(6.8, 0)}] (-1.2, 1.9) -- (0, 0.8);

\draw (-1.22, 1.4) node {$o$}; 
\draw (-0.75, 0.43) node {$w_{0}$}; 
\draw (-0.03, -0.22) node {$\w_{2\vartheta(1) - 2} o$};
\draw (1.44, 0.03) node {$\w_{2\vartheta(1) - 2}a_{1} o$};
\draw (0.57, 1.21) node {$\w_{2\vartheta(1) - 1} o$};
\draw (2.14, 2.1) node {$\w_{2\vartheta(1)} o$};

\draw (2.63, 0.8) node {$w_{1}$};
\draw (6.5, 1.42) node {$w_{N}$};

\draw (0.41, 0.31) node {$a_{1}$};
\draw (1.19, 0.62) node {$a_{1}$};
\draw (1.77, 1.32) node {$b_{1}^{2}$};

\draw (6.94, 0.57) node {$\w_{n} o = \w_{2\vartheta(N+1) - 2} o$};

\begin{scope}[shift={(3.4, 0)}]
\draw (-0.03, -0.22) node {$\w_{2\vartheta(2) - 2} o$};
\draw (1.44, 0.03) node {$\w_{2\vartheta(2) - 2}a_{2} o$};
\draw (0.57, 1.21) node {$\w_{2\vartheta(2) - 1} o$};
\draw (2.14, 2.1) node {$\w_{2\vartheta(2)} o$};

\draw (0.41, 0.31) node {$a_{2}$};
\draw (1.19, 0.62) node {$a_{2}$};
\draw (1.77, 1.32) node {$b_{2}^{2}$};
\end{scope}

\end{tikzpicture}
\caption{Words $w_{j}$'s, $a_{j}$'s and $b_{j}$'s that arise from a trajectory.}
\label{fig:trajectory}
\end{figure}
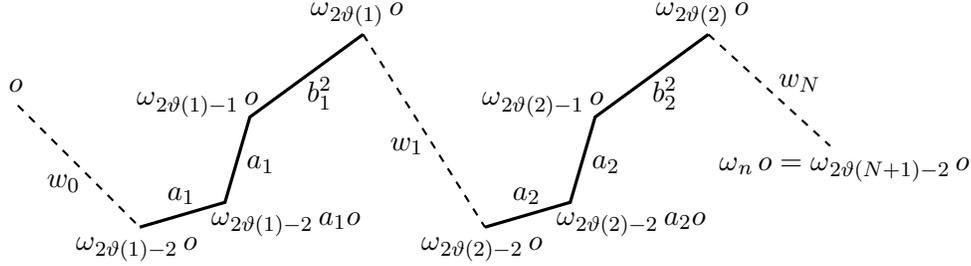

We now inductively define sets $P_{k}(\mathbf{g}) \subseteq \{1, \ldots, k\}$ and a moving point $z_{k}$ for $k = 0, \ldots, N$. First take $P_{0} := \emptyset$ and $z_{0} := o$. Now given $(P_{k-1}, z_{k-1})$, $(P_{k}, z_{k})$ is determined as follows (see Figure \ref{fig:Criterion}).\begin{enumerate}[label=(\Alph*)]
\item  If $a_{k} \neq b_{k}^{-1}$ and \begin{equation}\label{eqn:condA}\begin{aligned}
(z_{k-1}, \w_{2\vartheta(k) - 2} a_{k}^{t}o)_{\w_{2\vartheta(k) - 2} o} &< K \quad \textrm{for} \,\,t \in \{1, 2\}, \\
(\w_{2\vartheta(k) - 1} o, \w_{2\vartheta(k+1) - 2} o)_{\w_{2\vartheta(k)} o} &< K
\end{aligned}
\end{equation}
hold, then we set $P_{k} := P_{k-1} \cup \{k\}$ and $z_{k}: = \w_{2\vartheta(k) - 1} o$. Note that Inequality \ref{eqn:condA} is equivalent to \begin{equation}\label{eqn:varCondA}\begin{aligned}
(\w_{2\vartheta(k) - 2}^{-1} z_{k-1}, a_{k}^{t}o)_{o} &< K \quad \textrm{for} \,\,t \in \{1, 2\}, \\
(b_{k}^{-2} o, w_{k}o)_{o} &< K.
\end{aligned}
\end{equation}
\item If not, we seek sequences $\{i(1) < \cdots < i(M) \} \subseteq P_{k-1}$ with cardinality $M \ge 2$ such that: \begin{enumerate}[label=(\roman*)]
\item $[\w_{2\vartheta(i(1))- 1} o, \w_{2 \vartheta(i(M)) - 2} a_{i(M)} o]$ is fully $D(K, \epsilon)$-marked with $(\gamma_{j})_{j=1}^{M-1}$, $(\eta_{j})_{j=2}^{M}$, where \[\begin{aligned}
\gamma_{1} &= [\w_{2\vartheta(i(1)) - 1} o, \w_{2\vartheta(i(1))} o],\\
\gamma_{j} &=  [\w_{2\vartheta(i(j)) - 2} a_{i(j)} o, \w_{2\vartheta(i(j))- 1} o] \quad (2\le j \le M-1), \\
\eta_{j} &= [\w_{2\vartheta(i(j)) - 2}o, \w_{2\vartheta(i(j)) - 2} a_{i(j)} o] \quad (2 \le j \le M),
\end{aligned}
\]
\item $(\w_{2\vartheta(i(M)) - 2} a_{i(M)} o, \w_{2\vartheta(k+1) - 2} o)_{\w_{2\vartheta(i(M)) - 1}o} < K$.
\end{enumerate}
If exists, let $\{i(1) < \cdots < i(M)\}$ be such a sequence with maximal $i(1)$; we set $P_{k} := P_{k-1} \cap \{1, \ldots, i(1)\}$ and $z_{n}:= \w_{2\vartheta(i(M)) - 2} a_{i(M)} o$. If such a sequence does not exist, then we set $P_{k} := \emptyset$ and $z_{k} := o$.\footnote{When there are several sequences that realize maximal $i(1)$, we choose the maximum in the lexicographic order on the length of sequences and $i(2)$, $i(3)$, $\ldots$.}
\end{enumerate}

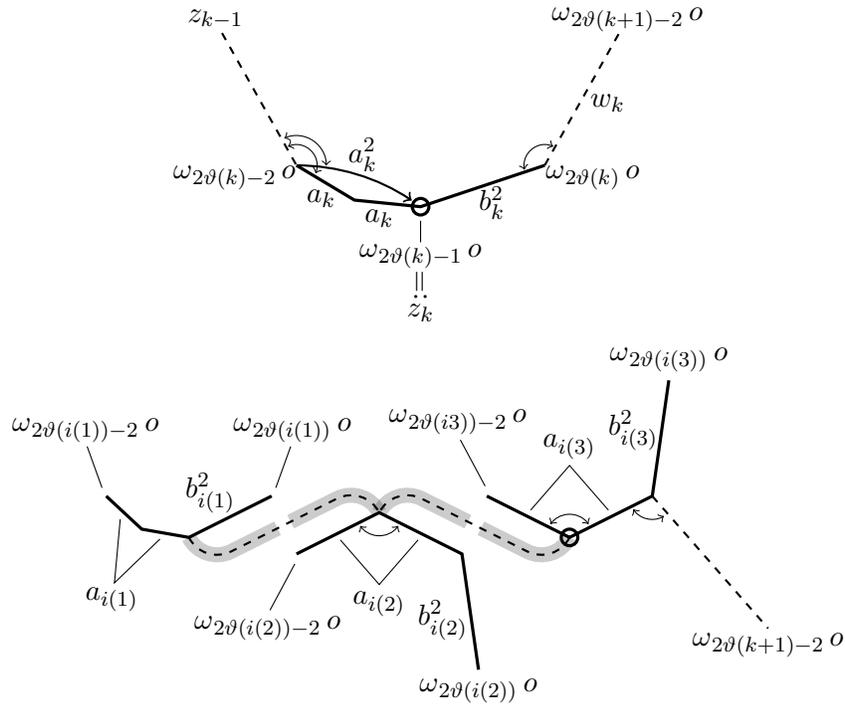
\begin{figure}
\centering
\begin{tikzpicture}
\def\c{1.1}
\draw[very thick] (0, 0) -- (0.7*\c, -0.42*\c) -- (1.5*\c, -0.5*\c) -- (3*\c, 0);
\draw[dashed, thick] (-0.9*\c, 1.6*\c) -- (0, 0);
\draw[dashed, thick] (3*\c, 0) -- (3.9*\c, 1.6*\c);
\draw[->, thick] (0, 0) arc (90: 90 - 36.869897645844021 +3: 2.5*\c);
\draw[<->] (-0.25*0.4067366430758*\c, 0.25*0.913545457642601*\c) arc (114:-22.4:0.25*\c);
\draw[<->, shift={(3*\c, 0)}, xscale=-1] (-0.25*0.4067366430758*\c, 0.25*0.913545457642601*\c) arc (114:-12:0.25*\c);
\draw[<->] (-0.36*0.4067366430758*\c, 0.36*0.913545457642601*\c) arc (114:-0.4:0.36*\c);

\draw (-0.98*\c, 1.77*\c) node {$z_{k-1}$};
\draw (-0.75*\c, -0.15*\c) node {$\w_{2\vartheta(k) - 2} o$};
\draw (0.3*\c, -0.4*\c) node {$a_{k}$};
\draw (1*\c, -0.64*\c) node {$a_{k}$};
\draw (0.8*\c, 0.13*\c) node {$a_{k}^{2}$};
\draw (1.5*\c, -1.1*\c) node {$\w_{2\vartheta(k) - 1}o$};
\draw (1.5*\c, -0.67*\c) -- (1.5*\c, -0.93*\c);
\draw (1.5*\c, -1.45*\c) node[rotate=90] {$:=$};
\draw (1.5*\c, -1.78*\c) node {$z_{k}$};
\draw (3.58*\c, -0.15*\c) node {$\w_{2\vartheta(k)}o$};
\draw (2.36*\c, -0.45*\c) node {$b_{k}^{2}$};
\draw (3.76*\c, 0.76*\c) node {$w_{k}$};
\draw (4*\c, 1.77*\c) node {$\w_{2\vartheta(k+1) - 2} o$};
\draw[very thick] (1.5*\c, -0.5*\c) circle (0.1*\c);

\begin{scope}[shift={(-1.3*\c, -4.5*\c)}]
\draw[very thick] (-\c, 0.5*\c) --  (-0.57*\c, 0.1*\c) -- (0, 0) -- (\c, 0.5*\c);
\draw[very thick] (1.3*\c, -0.2*\c) -- (2.3*\c, 0.3*\c) -- (3.3*\c, -0.2*\c) -- (3.5*\c, -1.6*\c);
\draw[very thick] (3.6*\c, 0.5*\c) -- (4.6*\c, 0*\c) -- (5.6*\c, 0.5*\c) -- (5.8*\c, 1.9*\c);

\draw[dashed, thick, rotate=26.565051177077989] (0, 0) arc (180:270:0.380131556174964*\c) -- (2.191346617949794*\c - 0.380131556174964*\c, -0.380131556174964*\c) arc (90:0: 0.380131556174964*\c);

\draw[dashed, thick, shift={(4.6*\c, 0)}, xscale=-1, rotate=26.565051177077989] (0, 0) arc (180:270:0.380131556174964*\c) -- (2.191346617949794*\c - 0.380131556174964*\c, -0.380131556174964*\c) arc (90:0: 0.380131556174964*\c);

\draw[opacity=0.2, line width=2mm, rotate=26.565051177077989] (0, 0) arc (180:270:0.380131556174964*\c) -- (0.46*2.191346617949794*\c, -0.380131556174964*\c);

\draw[opacity=0.2, line width=2mm, shift={(2.3*\c, 0.3*\c)}, rotate=180+26.565051177077989] (0, 0) arc (180:270:0.380131556174964*\c) -- (0.46*2.191346617949794*\c, -0.380131556174964*\c);

\draw[opacity=0.2, line width=2mm, shift={(2.3*\c, 0.3*\c)}, xscale=-1, rotate=180+26.565051177077989] (0, 0) arc (180:270:0.380131556174964*\c) -- (0.46*2.191346617949794*\c, -0.380131556174964*\c);

\draw[opacity=0.2, line width=2mm, shift={(4.6*\c, 0)}, xscale=-1, rotate=26.565051177077989] (0, 0) arc (180:270:0.380131556174964*\c) -- (0.46*2.191346617949794*\c, -0.380131556174964*\c);

\draw[dashed, thick] (5.6*\c, 0.5*\c) -- (7*\c, -1.1*\c);

\draw (-0.93*\c, -0.75*\c) node {$a_{i(1)}$};
\draw (-0.83*\c, 0.22*\c) -- (-0.9*\c, -0.55*\c) -- (-0.35*\c, -0.02*\c);
\draw (0.26*\c, 0.54*\c) node {$b_{i(1)}^{2}$};

\draw (2.3*\c, -0.8*\c) node {$a_{i(2)}$};
\draw (1.83*\c,  -0.03*\c) -- (2.3*\c, -0.57*\c) -- (2.77*\c, -0.03*\c);
\draw (3.07*\c, -1*\c) node {$b_{i(2)}^{2}$};

\draw (4.6*\c, 1.1*\c) node {$a_{i(3)}$};
\draw (4.13*\c,  0.33*\c) -- (4.6*\c, 0.87*\c) -- (5.07*\c, 0.33*\c);
\draw (5.37*\c, 1.3*\c) node {$b_{i(3)}^{2}$};

\draw[<->, shift={(2.3*\c, 0.3*\c)}, rotate=-35] (0.28*\c, 0) arc (0:-110:0.28*\c);
\draw[<->, shift={(4.6*\c, 0*\c)}, rotate=35] (0.28*\c, 0) arc (0:110:0.28*\c);
\draw[<->, shift={(5.6*\c, 0.5*\c)}, rotate=-60] (0.28*\c, 0) arc (0:-85:0.28*\c);

\draw (-1.25*\c, 1.3*\c) node {$\w_{2\vartheta(i(1)) - 2} o$};
\draw (1.25*\c, 1.3*\c) node {$\w_{2\vartheta(i(1))} o$};

\draw (-1.23*\c, 1.1*\c) -- (-1.05*\c, 0.6*\c);
\draw (1.23*\c, 1.1*\c) -- (1.05*\c, 0.6*\c);

\draw(0.95*\c, -1.1*\c) node {$\w_{2\vartheta(i(2)) - 2}o$};
\draw (3.5*\c, -1.85*\c) node {$\w_{2\vartheta(i(2))}o$};
\draw(0.98*\c, -0.88*\c) -- (1.28*\c, -0.32*\c);

\begin{scope}[shift={(2.3*\c, 0)}, shift={(0, 0.15*\c)}, yscale=-1, shift={(0, -0.15*\c)}]
\draw(0.95*\c, -1.1*\c) node {$\w_{2\vartheta(i3)) - 2}o$};
\draw (3.5*\c, -1.85*\c) node {$\w_{2\vartheta(i(3))}o$};
\draw(0.98*\c, -0.88*\c) -- (1.28*\c, -0.32*\c);
\end{scope}

\draw (7*\c, -1.3*\c) node {$\w_{2\vartheta(k+1) - 2} o$};
\draw[very thick] (4.6*\c, 0*\c) circle (0.1*\c);

\end{scope}

\end{tikzpicture}
\caption{Schematics for Criteria (A), (B) for the construction of $P_{k}$. The upper configuration describes the situation when $k$ is added in $P_{k}$. In the lower configuration, $\{i(1) < i(2) < i(3)\}$ satisfies items (i), (ii) in Criterion (B). Here, the shaded subsegments of the dashed lines fellow travel $\gamma_{1}$, $\eta_{2}$, $\gamma_{2}$ and $\eta_{3}$, from left to right, respectively. The newly chosen $z_{k}$ is highlighted by a circle.}
\label{fig:Criterion}
\end{figure}

Figure \ref{fig:pivoting} illustrates how $P_{k}$ evolves as $k$ increases. The path $\mathbf{g}$ under consideration has $\Theta(\mathbf{g}) = \{3, 6, 10, 13, 19\}$. Note that \[
\w_{0} o = o, \,\, \w_{2\vartheta(1) - 2} o = \w_{4} o,\,\,  \w_{2\vartheta(1) - 1} o=\w_{5} o, \,\, \w_{2\vartheta(1)} = \w_{6} o, \,\,  \w_{2\vartheta(2) - 2} o = \w_{10} o
\] are arranged as required in Criterion (A), which implies $P_{1}(\mathbf{g})= \{1\}$. Since $\w_{5} o$, $[\w_{10} o, \w_{11} o]$, $[\w_{11} o, \w_{12} o]$, $\w_{18} o$ are arranged as desired, $P_{2}(\mathbf{g}) = \{1, 2\}$ (even though $(\w_{11} o, \w_{i} o)_{\w_{12} o}$ is not always small for all $i > 12$). By a similar reason we have $P_{3}(\mathbf{g}) = \{1, 2, 3\}$. 

Since $\w_{36} o$ is on the left of $[\w_{25} o, \w_{26} o]$, however, $P_{4}(\mathbf{g})$ is not $\{1, 2, 3, 4\}$. If $\w_{36} o$ were on the right of $[\w_{24} o, \w_{25} o]$ at least, $P_{4}(\mathbf{g}) = \{1, 2, 3\}$ might have held; but it is not the case. Since $\w_{36} o$ is not on the right of $[\w_{18} o, \w_{19} o]$, $P_{4}(\mathbf{g}) = \{1, 2\}$ cannot hold either; we only have $P_{4}(\mathbf{g}) = \{1\}$. $P_{5}(\mathbf{g})$ then becomes $\{1, 5\}$. 

\begin{figure}
\centering
\begin{tikzpicture}
\def\c{0.5}
\draw (0, 0) -- (0.18*\c, 1.83*\c) -- (1.7*\c, 2.03*\c) -- (2.05*\c, 1*\c) -- (0.38*\c, -0.45*\c) --    (1.76*\c, -0.9*\c) -- (3*\c, 0) -- (5.24*\c, 1.3*\c) -- (3.87*\c, 2.84*\c) -- (5.9*\c, 3.07*\c) -- (5.8*\c, 0.2*\c) --    (7.78*\c, 0.86*\c) -- (9.25*\c, 0) -- (8.38*\c, -1.17*\c) -- (7.4*\c, -0.63*\c) -- (10.3*\c, 1.11*\c) -- (10.2*\c, -1.02*\c) -- (9*\c, 1.4*\c) -- (12.4*\c, 1.2*\c) --    (14.3*\c, 1.83*\c) -- (15.8*\c, 1.35*\c) -- (17.8*\c, 3*\c)  -- (18.5*\c, \c) -- (16.83*\c, 2.5*\c) -- (16.22*\c, -0.1*\c) --   (18*\c, -0.04*\c) -- (19.73*\c, 0.5*\c) -- (20.5*\c, -0.7*\c) -- (17.75*\c, -0.78*\c) -- (18.07*\c, -2.35*\c) -- (19.02*\c, -1.48*\c) -- (15.5*\c, -1.8*\c) -- (13.12*\c, -1.04*\c) -- (13.63*\c, -2.4*\c) -- (13.66*\c, -1.93*\c) -- (12*\c, -2.7*\c) -- (11.3*\c, 0.5*\c) -- (13*\c, -0.05*\c) -- (14.7*\c, -0.01*\c) -- (14.6*\c, -1.2*\c) -- (15.2*\c, -0.3*\c);
\draw[line width=0.7mm] (0.38*\c, -0.45*\c) --    (1.76*\c, -0.9*\c) -- (3*\c, 0);
\draw[line width=0.7mm] (5.8*\c, 0.2*\c) --    (7.78*\c, 0.86*\c) -- (9.25*\c, 0) ;
\draw[line width=0.7mm] (12.4*\c, 1.2*\c) --    (14.3*\c, 1.83*\c) -- (15.8*\c, 1.35*\c);
\draw[line width=0.7mm] (16.22*\c, -0.1*\c) --   (18*\c, -0.04*\c) -- (19.73*\c, 0.5*\c) ;
\draw[line width=0.7mm]   (11.3*\c, 0.5*\c) -- (13*\c, -0.05*\c) -- (14.7*\c, -0.01*\c)  ;

\draw[opacity=0.2, line width = 3.5mm, line cap=round]  (1.76*\c, -0.9*\c) -- (3*\c, 0) --  (5.8*\c, 0.2*\c) -- (5.8*0.67*\c +  7.78*0.33*\c, 0.2*0.65*\c + 0.86*0.35*\c);

\draw[opacity=0.14, line width = 3.5mm, line cap=round]  (5.8*0.33*\c +  7.78*0.67*\c, 0.2*0.35*\c + 0.86*0.65*\c) --   (7.78*\c, 0.86*\c) --  (11.3*\c, 0.5*\c) ;

\fill (0, 0) circle (0.06);
\draw (0, -0.25) node {$o$};
\fill (15.2*\c, -0.3*\c) circle (0.06);
\draw (15.2*\c +0.35, -0.3*\c - 0.25) node {$\w_{40}o$};

\end{tikzpicture}
\caption{An example of a sample path $\mathbf{g}$ with length 40. The vertices represent $\w_{i} o$ for $i = 0, \ldots, 40$; the thick segments represent Schottky progresses. The shaded region highlights the required witnessing by Schottky segments in Criterion (B) for $P_{4}(\mathbf{g})$.
}
\label{fig:pivoting}
\end{figure}
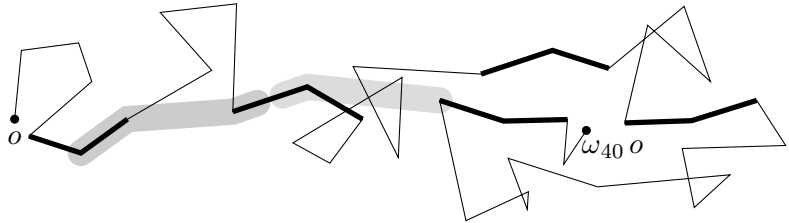

Note the following facts:
\begin{enumerate}
\item $P_{k}(\mathbf{g})$ is measurable with respect to the choice of $g_{i}$.
\item $i \in P_{m}$ only if $P_{i} = P_{i-1} \cup \{i\}$ and $i$ survives during stage $i+1, \ldots, m$.
\item If $i \in P_{m}$ and $i \in P_{n}$, then $\{1, \ldots, i\} \cap P_{m} = \{1, \ldots, i\} \cap P_{n}$.
\end{enumerate}

\begin{lem}\label{lem:intermediate}
The following holds for any $0 \le k \le n$ and $\mathbf{g} \in G^{n}$. Let $l< m$ be consecutive elements in $P_{k}= P_{k}(\mathbf{g})$, i.e., $l, m \in P_{k}$ and $l = \max (P_{k} \cap \{1, \ldots, m-1\})$. Let also $t \in \{1, 2\}$. Then there exists a sequence \[
\{l = i(1)< \ldots < i(M') =m\} \subseteq P_{k}
\] with cardinality $M' \ge 2$ such that $[\w_{2\vartheta(l) - 1} o, \w_{2\vartheta(m) - 2} a_{m}^{t} o]$ is fully $D(K, \epsilon)$-marked with $(\gamma_{j})_{j=1}^{M'-1}$, $(\eta_{j})_{j=2}^{M'}$, where \begin{equation}\label{eqn:intermediate}\begin{aligned}
\gamma_{1} &= [\w_{2\vartheta(i(1)) - 1} o, \w_{2\vartheta(i(1))} o],\\
\gamma_{j} &=  [\w_{2\vartheta(i(j)) - 2} a_{i(j)} o, \w_{2\vartheta(i(j))- 1} o] \quad (2 \le j \le M'-1), \\
\eta_{j} &= [\w_{2\vartheta(i(j)) - 2}o, \w_{2\vartheta(i(j)) - 2} a_{i(j)} o] \quad (2 \le j \le M'-1), \\ 
\eta_{M'} &= [\w_{2\vartheta(i(M)) - 2} o, \w_{2\vartheta(i(M')) - 2} a_{m}^{t} o].
\end{aligned}
\end{equation}
\end{lem}

\begin{proof}
$l, m \in P_{k}$ implies that $l \in P_{l}$ and $l, m \in P_{m}$. In particular, $l$ ($m$, resp.) is newly chosen at stage $l$ ($m$, resp.) by fulfilling Criterion (A). Hence, $(\w_{2\vartheta(l) - 1} o, \w_{2\vartheta(l+1) - 2} o)_{\w_{2\vartheta(l)} o} < K$ and $z_{l} = \w_{2\vartheta(l) - 1} o$. Moreover, $P_{m} = P_{m-1} \cup \{m\}$ and $l = \max P_{m-1}$.

If $l = m-1$, then $m$ is newly chosen at stage $m = l+1$. In this case we have \[(\w_{2\vartheta(l) - 1} o, \w_{2\vartheta(m) - 2} a_{m}^{t} o)_{\w_{2\vartheta(m) - 2} o} = (z_{l}, \w_{2\vartheta(m) - 2} a_{m}^{t} o)_{\w_{2\vartheta(m) - 2} o} < K
\] from Criterion (A) for $m$. Then Fact \ref{lem:1segment} implies that $[\w_{2\vartheta(l) - 1} o, \w_{2\vartheta(m) - 2}a_{m}^{t} o]$ is fully $D(K, \epsilon)$-marked with \[
[\w_{2\vartheta(l) - 1} o, \w_{2\vartheta(l)} o], \,\,[\w_{2\vartheta(m) - 2} o, \w_{2\vartheta(m) - 2} a_{m}^{t} o].
\]Hence, $\{l = i(1) < i(2) = m\}$ works.

If $l < m-1$, then $l=\max P_{m-1}$ has survived at stage $m-1$ by fulfilling Criterion (B). This means that there exist $l= i(1) < \ldots < i(M)$ in $P_{m-2}$  with $M  \ge 2$ such that $[\w_{2\vartheta(l) - 1} o, \w_{2\vartheta(i(M)) - 2}a_{i(M)} o]$ is fully $D(K, \epsilon)$-marked with $(\gamma_{j})_{j=1}^{M-1}$, $(\eta_{j})_{j=2}^{M}$. Here, $\gamma_{j}$, $\eta_{j}$'s are as in Equation \ref{eqn:intermediate}. Moreover, we have\begin{equation}\label{eqn:intermediateCase2}
(\w_{2\vartheta(i(M)) - 2} a_{i(M)} o, \w_{2\vartheta(m) - 2} o)_{\w_{2\vartheta(i(M)) - 1} o} < K
\end{equation}
and  $z_{m-1} = \w_{2\vartheta(i(M)) - 2} a_{i(M)} o$. 

We now claim that $i(1) < \ldots < i(M)$ and $i(M+1) := m$ together serve as the desired sequence (hence $M' = M+1 \ge 3$). First, since $m$ is newly chosen based on Criterion (A), we have \begin{equation}\label{eqn:intermediateCase2.2}(\w_{2\vartheta(i(M)) - 2} a_{i(M)} o, \w_{2\vartheta(m) - 2} a_{m}^{t})_{\w_{2\vartheta(m) - 2} o} = (z_{m-1}, \w_{2\vartheta(m) - 2} a_{m}^{t})_{\w_{2\vartheta(m) - 2} o} < K.
\end{equation}
Now combining Inequality \ref{eqn:intermediateCase2} and \ref{eqn:intermediateCase2.2}, Fact \ref{lem:1segment} implies that  $[\w_{2\vartheta(i(M)) - 2} a_{i(M)} o, \w_{2\vartheta(m) - 2} a_{m}^{t} o]$ is $D(K, \epsilon)$-witnessed by \[
( [\w_{2\vartheta(i(M)) - 2} a_{i(M)} o, \w_{2\vartheta(i(M)) - 1} o], [\w_{2\vartheta(m) - 2} o, \w_{2\vartheta(m) - 2} a_{m}^{t} o]) = (\gamma_{M}, \eta_{M}).
\] Finally, we have\[
(\w_{2\vartheta(i(M)) - 1} o, \w_{2\vartheta(i(M)) - 2} o)_{\w_{2\vartheta(i(M)) - 2} a_{i(M)} o} = (a_{i(M)} o, a_{i(M)}^{-1} o)_{o} <K<D(K, \epsilon)
\] thanks to Property (7) of the $(K, K', \epsilon)$-Schottky set $S_{1}$. Combining these implies that $[\w_{2\vartheta(l) - 1} o, \w_{2\vartheta(m) - 2} a_{m}^{t} o]$ is fully $D(K, \epsilon)$-marked with $(\gamma_{j})_{j=1}^{M}$, $(\eta_{j})_{j=2}^{M+1}$ that are as in Equation \ref{eqn:intermediate}.
\end{proof}

\begin{lem}\label{lem:extremal}
The following holds for any $n > 0$ and $\mathbf{g} \in G^{n}$. 
Let \[
\Theta(\mathbf{g}) = \{\vartheta(1) < \ldots < \vartheta(N)\} \quad \big(N = \#\Theta(\mathbf{g})\big)
\]
and \[
P_{N}(\mathbf{g})= \{\iota(1) < \ldots < \iota(m)\} \quad \big(m = \# P_{N}(\mathbf{g})\big).
\] Then there exist $M \ge m$, Schottky segments $(\gamma_{l})_{l=1}^{M}$, $(\eta_{l})_{l=1}^{M}$ and $1 \le l(1) < \ldots < l(m) \le M$ such that: \begin{enumerate}
\item  $[o, \w_{n} o]$ is $D(K, \epsilon)$-marked with $(\gamma_{l})_{l}$, $(\eta_{l})_{l}$, and 
\item $\gamma_{l(t)} = [\w_{2\vartheta(\iota(t)) - 1} o, \w_{2\vartheta(\iota(t))} o]$ and $\eta_{l(t)} = [\w_{2\vartheta(\iota(t)) - 2} o, \w_{2 \vartheta(\iota(t)) -1} o]$.
\end{enumerate}
\end{lem}

\begin{proof}
We will apply Lemma \ref{lem:markedConcat}. First recall Lemma \ref{lem:intermediate}: for each $t = 2, \ldots, m$, $[\w_{2\vartheta(\iota(t-1)) - 1} o, \w_{2\vartheta(\iota(t)) - 1} o]$ is fully $D(K, \epsilon)$-marked with some Schottky sequences $(\gamma_{l; t})_{l}$, $(\eta_{l;t})_{l}$'s, whose forms are given by Equation \ref{eqn:intermediate}. Here, note that the length of these sequences need not be 1; this leads to the possibility that $l(t) - l(t-1) > 1$.

Given the above result, it suffices to prove the following items:\begin{enumerate}
\item $[o, \w_{2\vartheta(\iota(1)) - 1} o]$ is fully $D(K, \epsilon)$-marked with $[o, o]$, $[\w_{2\vartheta(\iota(1)) - 2} o, \w_{2\vartheta(\iota(1)) - 1} o]$;
\item $[\w_{2\vartheta(\iota(m)) - 1} o, \w_{n} o]$ is fully $D(K, \epsilon)$-marked with some sequences $(\gamma'_{j})_{j=1}^{M'}$, $(\eta'_{j})_{j=2}^{M'+1}$ of Schottky segments, where $\gamma'_{1} = [
\w_{2\vartheta(\iota(m)) - 1} o, \w_{2\vartheta(\iota(m))} o]$.
\item $(\w_{2\vartheta(\iota(t)) - 2} o, \w_{2\vartheta(\iota(t))} o)_{\w_{2\vartheta(\iota(t)) - 1} o} < D(K, \epsilon)$ for each $t=1, \ldots, m$.
\end{enumerate}

First, $\iota(1) = \min P_{N}$ implies that $P_{\iota(1) - 1} = \emptyset$, $z_{\iota(1) - 1} = o$ and that $\iota(1)$ is newly chosen at stage $\iota(1)$. Hence, we have \[
(o, \w_{2\vartheta(\iota(1)) - 1} o)_{\w_{2\vartheta(\iota(1)) - 2} o} = (z_{\iota(1)-1}, \w_{2\vartheta(\iota(1)) - 1} o)_{\w_{2\vartheta(\iota(1)) - 2} o} < K
\] and Fact \ref{lem:1segment} implies the first item.

Next, we observe how $\iota(m)$ survived in $P_{N}$. If $\iota(m) = N$, then it was newly chosen at stage $N$; $(\w_{2\vartheta(\iota(m)) - 1}o, \w_{n} o)_{\w_{2\vartheta(\iota(m))} o} < K$ holds and Fact \ref{lem:1segment} implies that $[\w_{2\vartheta(\iota(m)) - 1} o, \w_{n} o]$ is fully $D(K, \epsilon)$-marked with \[
[w_{2\vartheta(\iota(m)) - 1} o, w_{2\vartheta(\iota(m))} o], \quad [\w_{n} o, \w_{n} o].
\]
If $\iota(m) \neq N$, then it has survived at stage $N$ by fulfilling Criterion (B). Thus, there exist $i \in P_{N-1}$ such that $(\w_{2\vartheta(i) - 2}a_{i} o, \w_{n} o)_{\w_{2\vartheta(i) - 1} o } < K$ and Schottky segments $(\gamma_{j}')_{j=1}^{M'-1}$, $(\eta_{j}')_{j=2}^{M'}$ such that $[\w_{2\vartheta(\iota(m)) - 1} o, \w_{2\vartheta(i) - 2} a_{i} o]$ is fully $D(K, \epsilon)$-marked with $(\gamma_{j}')$, $(\eta_{j}')$, where \[\gamma_{1}' = [\w_{2\vartheta(\iota(m)) - 1} o, \w_{2\vartheta(\iota(m))} o], \quad \eta_{M'}' = [\w_{2\vartheta(i) - 2} o, \w_{2\vartheta(i) - 2} a_{i} o].
\] Futhermore, the second item of Criterion (B) and Fact \ref{lem:1segment} imply that $[\w_{2\vartheta(i) - 2} a_{i} o, \w_{n} o]$ is $D(K, \epsilon)$-witnessed by \[
([\w_{2\vartheta(i) - 2} a_{i} o, \w_{2\vartheta(i) - 1} o], [\w_{n}o, \w_{n}o]).
\] Finally, recall that  \[
(\w_{2\vartheta(i) - 2} o, \w_{2\vartheta(i) - 1} o)_{\w_{2\vartheta(i) - 2} a_{i} o} = (a_{i}^{-1}o, a_{i} o)_{o} < K<D(K, \epsilon)
\] by Property (7) of the $(K, K', \epsilon)$-Schottky set $S_{1}$. Combining these, we conclude that $[\w_{2\vartheta(\iota(m)) - 1} o, \w_{n} o]$ is fully $D(K, \epsilon)$-marked with \[
(\gamma_{1}', \ldots, \gamma_{M'-1}', [\w_{2\vartheta(i) - 2} a_{i} o, \w_{2\vartheta(i) - 1} o]), \quad (\eta_{2}', \ldots, \eta_{M'}', [\w_{n} o, \w_{n}o]).
\]
This settles the second item.

For the third item let $t \in \{1, \ldots, m\}$. Since $\iota(t) \in P_{N}(\mathbf{g})$, $\iota(t)$ was newly chosen at stage $\iota(t)$ by fulfilling Criterion (A); hence, $a_{\iota(t)}^{-1} \neq b_{\iota(t)}$ and we deduce $(\w_{2\vartheta(\iota(t)) - 2} o, \w_{2\vartheta(\iota(t))} o)_{\w_{2\vartheta(\iota(t)) - 1} o}= (a_{\iota(t)}^{-2} o, b_{\iota(t)}^{2} o)_{o} < K$ from Properties (5), (6), (7) of Schottky sets. 
\end{proof}

The same proof also yields the following lemma: 

\begin{lem}\label{lem:intermediate2}
Let $k<k'$ be elements of $P_{N}(\mathbf{g})$ and $t \in \{1, 2\}$. Then there exist some sequences $(\gamma_{l})_{l=1}^{M-1}$, $(\eta_{l})_{l=2}^{M}$ of Schottky segments such that: \begin{enumerate}
\item  $[\w_{2\vartheta(k) - 1} o, \w_{2\vartheta(k') - 2} a_{k'}^{t} o]$ is fully $D(K, \epsilon)$-marked with $(\gamma_{l})$, $(\eta_{l})$, and 
\item $\gamma_{1} = [\w_{2\vartheta(k) - 1} o, \w_{2\vartheta(k)} o]$ and $\eta_{M} = [\w_{2\vartheta(k') - 2} o, \w_{2\vartheta(k') - 2} a_{k'}^{t} o]$.
\end{enumerate}
\end{lem}

Having established the properties of $P_{N}(\mathbf{g})$, our next goal is to estimate the size of $P_{N}(\mathbf{g})$. We first fix $N$, $\Theta = \{\vartheta(1) < \ldots < \vartheta(N)\}$ and the choices $(g_{2j-1}, g_{2j}) \notin \left(S_{1}^{(2)} \cup S_{1}^{(-2)} \right)^{2}$ for $j \notin \Theta$. Conditioned on these choices, we draw $(g_{2\theta(1) - 1}, g_{2\theta(1)}, \ldots, g_{2\theta(N) - 1}, g_{2\theta(N)})$ from $ \left(S_{1}^{(2)} \cup S_{1}^{(-2)} \right)^{2N}$ with the product measure of the uniform measure on $S_{1} \cup S_{1}^{-1}$. In particular, we regard $\mathbf{g}$, $\w_{j}$'s and $P_{k}$ as RVs of $s = (a_{1}, b_{1}, \ldots, a_{N}, b_{N})$; here, $a_{i}, b_{i}$ are independently drawn from $S_{0}$ with the uniform measure.

We will often modify the given choice $s$; the modified choices will be denoted by $\tilde{s} = (\tilde{a}_{1}, \ldots, \tilde{b}_{N})$ or $\bar{s} = (\bar{a}_{1}, \ldots, \bar{b}_{N})$. We will then denote by $\tilde{\w}_{j}$ or $\bar{\w}_{j}$ the sample path arising from the modified choices, respectively.

\begin{lem}\label{lem:0thCasePivot}
For $0 < k \le N$ and partial choices $s \in S_{0}^{2(k-1)}$, we have \[
\Prob\Big( \# P_{k}(s, a_{k}, b_{k}) = \# P_{k-1}(s) + 1\Big)\ge 9/10.
\]
\end{lem}

\begin{proof}
Recall Criterion (A) for $\#P_{k} = \# P_{k-1} + 1$. Note that the condition \begin{equation}\label{eqn:pivotCondition1}
(\w_{2\vartheta(k) - 1} o, \w_{2\vartheta(k+1) - 2} o)_{\w_{2\vartheta(k)} o} = (b_{k}^{-2} o, w_{k}o)_{o} < K
\end{equation} 
depends only on $b_{k}$ and not on other $a_{i}$'s and $b_{i}$'s. This holds for at least $(\#S_{1} - 1)$ choices of $b_{k}$ in $S_{1}$ and $(\#S_{1} - 1)$ choices in $S_{1}^{(-1)}$.

Let us now fix a choice $b_{k} \in S_{1}^{(\pm)}$ satisfying Condition \ref{eqn:pivotCondition1} and $a_{1}, b_{1}, \ldots, a_{k-1}, b_{k-1}$ that determine $\w_{2\vartheta(k) - 2}$ and $z_{k-1}$. Then the remaining conditions \begin{equation}\label{eqn:pivotCondition2}
\begin{aligned}
(z_{k-1}, \w_{2\vartheta(k) - 2} a_{k}o)_{\w_{2\vartheta(k) - 2} o} &= (\w_{2\vartheta(k) - 2}^{-1}z_{k-1}, a_{k}o)_{o}<K, \\
(z_{k-1}, \w_{2\vartheta(k) - 2} a_{k}^{2}o)_{\w_{2\vartheta(k) - 2} o} &= (\w_{2\vartheta(k) - 2}^{-1}z_{k-1}, a_{k}^{2}o)_{o}<K, \\ 
a_{k} &\neq b_{k}^{-1}
\end{aligned}
\end{equation} 
hold for at least $(\#S_{1} - 1)$ choices of $a_{k}$ in $S_{1}^{(\pm 1)}$ and $(\#S_{1} - 2)$ choices in $S_{1}^{(\mp 1)}$, due to Property (5), (6) and (7) of Schottky sets. Since Condition \ref{eqn:pivotCondition1} and Condition \ref{eqn:pivotCondition2} together constitute Criterion (A), we obtain \[
\Prob\left( \# P_{k} = \# P_{k-1}+1 \right) \ge \frac{2\#S_{1} - 2}{2\#S_{1}} \cdot  \frac{2\#S_{1} - 3}{2\#S_{1}} \ge 0.9. \qedhere
\]
\end{proof}

Given $a_{1}, b_{1}, \ldots, a_{k-1}, b_{k-1}$ and $b_{k}$, we define the set $\tilde{S}'_{k}$ of elements $a_{k}$ in $S_{0}$ that satisfy Condition \ref{eqn:pivotCondition2}. In the proof above, we have observed that  $ \# \left[ S_{0} \setminus \tilde{S}'_{k} \right]\le 3$.

\begin{lem} \label{lem:pivotEquiv}
Let $i \in P_{k}(s)$ for a choice $s = (a_{1}, b_{1}, \cdots, a_{N}, b_{N})$, and $\bar{s}$ be obtained from $s$ by replacing $a_{i}$ with $\bar{a}_{i} \in \tilde{S}'_{i}(a_{1}, b_{1}, \cdots, a_{i-1}, b_{i-1}, b_{i})$. Then $P_{l}(s) = P_{l}(\bar{s})$ and $\tilde{S}'_{l}(s) = \tilde{S}'_{l}(\bar{s})$ for any $1 \le l \le k$.
\end{lem}

\begin{proof}
Since $a_{1}, b_{1}, \ldots, a_{i-1}, b_{i-1}$ are intact, $P_{l}(s) = P_{l}(\bar{s})$ and $\tilde{S}'_{l}(s) = \tilde{S}'_{l}(\bar{s})$ hold for $l=0, \ldots, i-1$. At stage $i$, $b_{i}$ satisfies Condition \ref{eqn:pivotCondition1} (since $i \in P_{k}(s)$) and $\bar{a}_{i}$ satisfies Condition \ref{eqn:pivotCondition2}; hence, $i \in P_{k}(\bar{s})$ and $P_{i}(s) = P_{i}(\bar{s})$. We also have $\tilde{S}_{i}'(s) = \tilde{S}_{i}'(\bar{s})$.
At this stage, however, we have \[
\bar{z}_{i} = \bar{\w}_{2\vartheta(i)-1} o = g \w_{2\vartheta(i) - 1} o =gz_{i} \]
where \[
g :=  \w_{2\vartheta(i) - 1}\bar{a}_{i}^{2} (\w_{2\vartheta(i) - 1} a_{i}^{2})^{-1}.
\] More generally, we have \begin{equation}
\bar{\w}_{j} = g\w_{j} \quad (j \ge 2\vartheta(i) - 1),
\end{equation} or in other words, \begin{equation}\label{eqn:pivotEquivEqn1}
\begin{aligned}
\bar{\w}_{2\vartheta(j) - 1} &= g\w_{2\vartheta(j) - 1} \quad (j \ge i), \\
\bar{\w}_{2\vartheta(j) - 2} &= g\w_{2\vartheta(j) - 2}\quad (j > i).
\end{aligned}
\end{equation}
Recall again that the intermediate words $w_{j}$'s in between Schottky steps are unchanged.

We now claim the following for $i < l \le k$: \begin{enumerate}
\item If $s$ fulfills Criterion (A) at stage $l$, then so does $\bar{s}$.
\item If not and $\{i(1) < \ldots < i(M)\} \subseteq P_{l-1}(s)$ is the maximal sequence for $s$ in Criterion (B) at stage $l$, then it is also the maximal one for $\bar{s}$ at stage $l$.
\item In both cases, we have $P_{l}(s) = P_{l}(\bar{s})$ and $\bar{z}_{l} = g z_{l}$.
\end{enumerate}
Assuming the third item for $l-1$: $P_{l-1}(s) = P_{l-1}(\bar{s})$ and $\bar{z}_{l-1} = g z_{l-1}$, let us test Inequality \ref{eqn:varCondA} for $\bar{\w}$ in Criterion (A). If $s$ fulfills Criterion (A) at stage $l$, then \[\big(\bar{\w}_{2\vartheta(l) - 2}^{-1} \bar{z}_{l-1}, a_{l}^{t} o\big)_{o} = \left( \w_{2\vartheta(l) - 2}^{-1} g^{-1} \cdot g z_{l-1}, \,\,a_{l}^{t} o\right)_{o}  < K
\]
for $t = 1, 2$ and $(b_{l}^{-2} o, w_{l} o )_{o} < K$. Hence we obtain the first item. In this case we also deduce \[\begin{aligned}
P_{l}(s) &= P_{l-1}(s) \cup \{l\} = P_{l-1}(\bar{s}) \cup \{l\} = P_{l}(\bar{s}), \\
\bar{z}_{l} &= \bar{\w}_{2\vartheta(l) - 1} o = g \w_{2\vartheta(l) - 1} o = gz_{l},
\end{aligned}
\] which constitute the third item for $l$.

Let us now check the second item. Due to Equality \ref{eqn:pivotEquivEqn1}, a sequence $\{i(1) < \ldots < i(M)\}$ in $P_{l-1}(s) \cap \{i, \ldots, l-1\} = P_{l-1}(\bar{s})\cap \{i, \ldots, l-1\}$ works for $s$ in Criterion (B) if and only if it works for $\bar{s}$. Furthermore, $i$ belongs to $P_{l}(s)$ since $i \in P_{k}(s)$ and $i<l \le k$; hence, such sequences exist and the maximal sequence is chosen among them. Therefore, the maximal sequence $\{i(1) < \ldots < i(M)\}$ (with cardinality $M \ge 2$) for $s$ is also maximal for $\bar{s}$. We then deduce \[\begin{aligned}
P_{l}(s) &= P_{l-1}(s) \cap \{1, \ldots, i(1)\} = P_{l-1}(\bar{s}) \cap \{1, \ldots, i(1)\} = P_{l}(\bar{s}), \\
\bar{z}_{l} &= \bar{\w}_{2\vartheta(i(M)) - 2} a_{i(M)} o  = g\w_{2\vartheta(i(M)) - 2} a_{i(M)} o = gz_{l},
\end{aligned}
\]which constitute the third item for $l$. Here we used the condition $M \ge 2$ and $i(M) > i$; beware that $\bar{\w}_{2\vartheta(i) - 2} \bar{a}_{i} o$ and $g\w_{2\vartheta(i) - 2} a_{i} o$ may differ.

Since we have the base case $\bar{z}_{i} = gz_{i}$, an induction shows that $P_{l}(s) = P_{l}(\bar{s})$ for each $i<l\le k$. Moreover, Equality \ref{eqn:pivotEquivEqn1} and $\bar{z}_{l-1} = gz_{l-1}$ imply that $\tilde{S}_{l}'(s) = \tilde{S}_{l}'(\bar{s})$.
\end{proof}

Given $1 \le k \le N$ and a partial choice $s = (a_{1}, b_{1}, \ldots, a_{k}, b_{k})$, we say that $\bar{s} = (\bar{a}_{1}, \bar{b}_{1}, \ldots, \bar{a}_{k}, \bar{b}_{k})$ is \emph{pivoted from $s$} if: \begin{enumerate}
\item $b_{j} = \bar{b}_{j}$ for all $1\le j \le k$, 
\item $\bar{a}_{i} \in \tilde{S}'_{i}(s)$ for $i \in P_{k}(s)$, and 
\item $a_{j} = \bar{a}_{j}$ for all other $j \notin P_{k}(s)$. 
\end{enumerate}
Lemma \ref{lem:pivotEquiv} then asserts that being pivoted from each other is an equivalence relation. For each $s \in S_{0}^{2k}$, let $\mathcal{E}_{k}(s)$ be the equivalence class of $s$.

\begin{lem} \label{lem:pivotCondition}
For $0 \le k < N$, $j \ge0$ and $s \in S_{0}^{2k}$, we have \[
\Prob\Big(\# P_{k+1}(\tilde{s}, a_{k+1}, b_{k+1}) < \# P_{k}(s) - j \, \Big| \, \tilde{s} \in \mathcal{E}_{k}(s), (a_{k+1}, b_{k+1}) \in S_{0}^{2}\Big) \le 1/10^{j+1}.
\]
\end{lem}

\begin{proof}
Let us fix $s = (a_{1}, b_{1}, \ldots, a_{k}, b_{k})\in S_{0}^{2k}$ and \[
\mathcal{A} := \Big\{(a_{k+1}, b_{k+1}) \in S_{0}^{2}: \#P_{k+1}(s, a_{k+1}, b_{k+1}) = \#P_{k}(s)+1 \Big\}.
\] Then Lemma \ref{lem:0thCasePivot} implies that $\Prob(\mathcal{A} \,|\, S_{0}^{2}) \ge 0.9$. Moreover, for $(a_{k+1}, b_{k+1}) \in \mathcal{A}$ and $\tilde{s} \in \mathcal{E}_{k}(s)$, $(\tilde{s}, a_{k+1}, b_{k+1})$ is pivoted from $(s, a_{k+1}, b_{k+1})$ since $(\tilde{s}, a_{k+1}, b_{k+1})$ and $(s, a_{k+1}, b_{k+1})$ differ at slots in $P_{k}(s) \subseteq P_{k+1}(s, a_{k+1}, b_{k+1})$. Lemma \ref{lem:pivotEquiv} then implies that $P_{k+1}(\tilde{s}) = P_{k+1}(s) = P_{k}(s) \cup \{k+1\} = P_{k}(\tilde{s}) \cup \{k+1\}$. In other words, \[
\Prob\Big(\#P_{k+1}(\tilde{s}, a_{k+1}, b_{k+1}) < \#P_{k}(\tilde{s}) \, \Big| \, (a_{k+1}, b_{k+1}) \in S_{0}^{2} \Big) \le 1 - \Prob(\mathcal{A}) \le 1/10
\]
for each $\tilde{s} \in \mathcal{E}_{k}(s)$. Gathering all the cases, we deduce \[
\Prob\Big( \#P_{k+1}(\tilde{s}, a_{k+1}, b_{k+1}) < \#P_{k}(\tilde{s}) \, \Big| \, \tilde{s} \in \mathcal{E}_{k}(s), (a_{k+1}, b_{k+1}) \in S_{0}^{2} \Big) \le 1/10.
\]
This settles the case $j = 0$.

Now let $j = 1$. The event under discussion becomes void when $\# P_{k}(s) \le 2$. Excluding such cases, let $l< m$ be the last 2 elements of $P_{k}(s)$. For each choice $\tilde{s}$ in $\mathcal{E}_{k}(s)$ and each subset $A$ of $\tilde{S}_{m}'(s)$, we define \[
E(\tilde{s}, A) := \left\{ \bar{s} = (\bar{a}_{i}, \bar{b}_{i})_{i=1}^{k} : \bar{b}_{i} = \tilde{b}_{i} \,\,\textrm{for all}\,\, i, \,\,\bar{a}_{i} = \tilde{a}_{i}\,\,\textrm{for}\,\, i \neq m,\,\, \bar{a}_{m} \in A \right\}.
\]
In plain words, $E(\tilde{s}, A)$ is a set of choices that are pivoted from $\tilde{s}$ only at stage $m$, such that the pivotal choice belongs to $A$. Then $\{E(\tilde{s}, \tilde{S}_{m}'(s)) : \tilde{s} \in \mathcal{E}_{k}(s)\}$ partitions $\mathcal{E}_{k}(s)$ by Lemma \ref{lem:pivotEquiv}.

We now fix $(a_{k+1}, b_{k+1})\in S_{0}^{2}$ and $\tilde{s} = (\tilde{a}_{1}, \ldots, \tilde{b}_{k})\in \mathcal{E}_{k}(s)$. Let $A' \subseteq \tilde{S}_{m}'(s)$ be the collection of elements $\bar{a}_{m}$ that satisfies \begin{equation}\label{eqn:addCondPivot}
\big(\bar{a}_{m}^{-1} o, \,(\tilde{\w}_{2\vartheta(m) - 1})^{-1} \tilde{\w}_{2\vartheta(k) - 2} a_{k+1}^{2} b_{k+1}^{2} w_{k+1} o\big)_{o} = \big(\bar{a}_{m}^{-1} o,\, \tilde{b}_{m}^{2}  w_{m} \cdots \tilde{a}_{k}^{2} \tilde{b}_{k}^{2} w_{k} a^{2}_{k+1}  b^{2}_{k+1} w_{k+1} o)_{o} < K.
\end{equation}
Note that $A'$ depends on $\tilde{s}$, $a_{k+1}$ and $b_{k+1}$. By Property (5), (6) of Schottky sets, we have $\# \left[\tilde{S}'_{m}(s) \setminus A' \right] \le 2$.

We now claim that $\# P_{k+1}(\bar{s}, a_{k+1}, b_{k+1}) \ge \# P_{k}(s) - 1$ for $\bar{s} \in E(\tilde{s}, A')$. First, since $l<m$ are consecutive elements in $P_{k}(\bar{s})$, Lemma \ref{lem:intermediate} gives a sequence $\{l = i(1) < \ldots < i(M) = m\} \subseteq P_{k}$ with $M \ge 2$ such that $[\bar{\w}_{2\vartheta(l) - 1} o, \bar{\w}_{2\vartheta(m) - 2} \bar{a}_{m} o]$ is fully $D(K, \epsilon)$-marked with $(\gamma_{j})_{j=1}^{M-1}$, $(\eta_{j})_{j=2}^{M}$, where \[\begin{aligned}
\gamma_{1} &= [\bar{\w}_{2\vartheta(i(1)) - 1} o, \bar{\w}_{2\vartheta(i(1))} o],\\
\gamma_{j} &=  [\bar{\w}_{2\vartheta(i(j)) - 2} a_{i(j)} o, \bar{\w}_{2\vartheta(i(j))- 1} o] \quad (2 \le j \le  M-1), \\
\eta_{j} &= [\bar{\w}_{2\vartheta(i(j)) - 2}o, \bar{\w}_{2\vartheta(i(j)) - 2} a_{i(j)} o] \quad (2\le j \le  M).
\end{aligned}
\]
Moreover, Condition \ref{eqn:addCondPivot} implies that \[
(\bar{\w}_{2\vartheta(i(M)) - 2} a_{i(M)} o, \bar{\w}_{2\vartheta(k+1) - 2} o)_{\bar{\w}_{2\vartheta(i(M)) - 1}o} < K.
\]
In summary, $\{l=i(1) < \cdots < i(M)\} \subseteq P_{k}(\bar{s})$ works for $\bar{s}$ in Criterion (B) at stage $k+1$, which implies $P_{k+1}(\bar{s}) \supseteq P_{k}(\bar{s}) \cap \{1, \ldots, l\}$, hence the claim.

As a result, we deduce \[\begin{aligned}
\Prob\Big(\#P_{k+1}(\bar{s}, a_{k+1}, b_{k+1}) < \#P_{k}(s) - 1 \, \Big| \, \bar{s} \in E(\tilde{s}, \tilde{S}'_{m})\Big) &\le \frac{\# \left[ \tilde{S}'_{m}(s) \setminus A'\right]}{\# \tilde{S}'_{m}(s)}\\& \le \frac{2}{\# S_{0} - 3} \le 0.1.
\end{aligned}
\]
for any $\tilde{s} \in \mathcal{E}_{k}(s)$ and $(a_{k+1}, b_{k+1}) \in S_{0}^{2}$. Since $E(\tilde{s}, \bar{S}'_{m})$'s for $\tilde{s} \in \mathcal{E}_{k}(s)$ partition $\mathcal{E}_{k}(s)$, we deduce \[
\Prob\Big(\# P_{k+1}(\tilde{s}, a_{k+1}, b_{k+1}) < \# P_{k}(s) - 1 \, \Big| \, \tilde{s} \in \mathcal{E}_{k}(s) \Big) \le 0.1
\]
for any $(a_{k+1}, b_{k+1}) \in S_{0}^{2}$. Moreover, the above probability vanishes when $(a_{k+1}, b_{k+1}) \in \mathcal{A}$. Since $\Prob(\mathcal{A} \,|\, S_{0}^{2}) \ge 0.9$, we deduce \begin{equation}\label{eqn:firstPivotEst}
\Prob\Big(\#P_{k+1}(\tilde{s}, a_{k+1}, b_{k+1}) < \#P_{k}(s) - 1 \, \Big| \, \tilde{s} \in \mathcal{E}_{k}(s), (a_{k+1}, b_{k+1}) \in S_{0}^{2}\Big) \le 0.01.
\end{equation}
This settles the case $j=1$.

For $j =2$, we similarly assume $\#P_{k}(s) \ge 3$ and let $l' < l<m$ be the last 3 elements. We define the set $\mathcal{A}_{1}$ of $(\bar{a}_{m}, a_{k+1}, b_{k+1})$ in $\tilde{S}_{m}'(s) \times S_{0}^{2}$ such that \[
\#P_{k+1}(\underbrace{a_{1}, b_{1}, \ldots, \bar{a}_{m},  b_{m}, \ldots, a_{k}, b_{k}}_\text{obtained from $s$ by replacing $a_{m}$ with $\bar{a}_{m}$}, a_{k+1}, b_{k+1}) \ge \#P_{k}(s)-1,
\]
or equivalently, \[
P_{k}(s) \cap \{1, \ldots, l\} \subseteq P_{k+1}(a_{1}, b_{1}, \ldots, \bar{a}_{m}, b_{m}, \ldots, a_{k}, b_{k}, a_{k+1}, b_{k+1}).
\]
Now, if $\tilde{s} = (\tilde{a}_{1}, \tilde{b}_{1}, \ldots, \tilde{a}_{k}, \tilde{b}_{k}) \in \mathcal{E}_{k}(s)$ is such that $(\tilde{a}_{m}, a_{k+1}, b_{k+1}) \in \mathcal{A}_{1}$, then $(\tilde{s}, a_{k+1}, b_{k+1})$ is pivoted from $(a_{1}, b_{1}, \ldots, \tilde{a}_{m}, b_{m}, \ldots, a_{k+1}, b_{k+1})$ since they only differ at $a_{i}$'s for $i \in P_{k}(s) \cap \{1, \ldots, l\} \subseteq P_{k+1}(a_{1}, b_{1}, \ldots, \tilde{a}_{m}, b_{m}, \ldots, a_{k+1}, b_{k+1})$. Lemma \ref{lem:pivotEquiv} then implies that $P_{k+1}(\tilde{s}, a_{k+1}, b_{k+1})$ also contains $P_{k}(s) \cap \{1, \ldots, l\}$. This fact and Inequality \ref{eqn:firstPivotEst} implies that \[\begin{aligned}
&\Prob(\mathcal{A}_{1}\, | \, \tilde{S}'_{m}(s) \times S_{0}^{2}) \\
&= \Prob\Big(  (\tilde{a}_{m}, a_{k+1}, b_{k+1}) \in \mathcal{A}_{1}\, | \, \tilde{s} \in \mathcal{E}_{k}(s), (a_{k+1}, b_{k+1}) \in S_{0}^{2}\Big) \\
& \ge \Prob\Big(\#P_{k+1}(\tilde{s}, a_{k+1}, b_{k+1}) \ge \#P_{k}(s) - 1 \, \Big| \, \tilde{s} \in \mathcal{E}_{k}(s), (a_{k+1}, b_{k+1}) \in S_{0}^{2}\Big) \ge 0.99.
\end{aligned}
\]
We now define for $\tilde{s} \in \mathcal{E}_{k}(s)$ and each $A \subseteq \tilde{S}_{l}'(s)$ \
\[
E_{1}(\tilde{s}, A) := \left\{ \bar{s} = (\bar{a}_{i}, \bar{b}_{i})_{i=1}^{k} : \bar{b}_{i} = \tilde{b}_{i} \,\,\textrm{for all}\,\, i, \,\,\bar{a}_{i} = \tilde{a}_{i}\,\,\textrm{for}\,\, i \neq l,\,\, a_{l} \in A \right\}.
\]
Then $\{E(\tilde{s}, \tilde{S}_{l}'(s)) : \tilde{s} \in \mathcal{E}_{k}(s)\}$ partitions $\mathcal{E}_{k}(s)$ by Lemma \ref{lem:pivotEquiv}.

Now fixing $(a_{k+1}, b_{k+1}) \in S_{0}^{2}$ and  $\tilde{s} \in \mathcal{E}_{k}(s)$, let $A_{1}' \subseteq \tilde{S}_{l}'(s)$ be the collection of elements $\bar{a}_{l} \in \tilde{S}_{l}'(s)$ that satisfies\begin{equation}\label{eqn:addCondPivot2}
\big(\bar{a}_{l}^{-1} o, \,(\tilde{\w}_{2\vartheta(l) - 1})^{-1} \tilde{\w}_{2\vartheta(k) - 2}a_{k+1}^{2} b_{k+1}^{2} w_{k+1} o\big)_{o} =  (\bar{a}_{l}^{-1} o,\, \tilde{b}_{l}^{2} w_{l} \cdots \tilde{a}_{k}^{2} \tilde{b}_{k}^{2} w_{k} a^{2}_{k+1}  b^{2}_{k+1} w_{k+1} o)_{o} < K.
\end{equation}
By Property (5), (6) of Schottky sets, we have $\# \left[\tilde{S}'_{m}(s) \setminus A_{1}' \right] \le 2$.

We now claim that $\#P_{k+1}(\bar{s}, a_{k+1}, b_{k+1}) \ge \# P_{k}(s) - 2$ for $\bar{s} \in E_{1}(\tilde{s}, A_{1}')$. First, since $l' < l$ are consecutive elements in $P_{k}(\bar{s})$, Lemma \ref{lem:intermediate} gives a sequence $\{l' = i(1) < \ldots < i(M) = l\} \subseteq P_{k}$ such that $[\bar{\w}_{2\vartheta(l') - 1} o, \bar{\w}_{2\vartheta(l) - 2} \bar{a}_{l} o]$ is fully $D(K, \epsilon)$-marked with $(\gamma_{j})_{j=1}^{M-1}$, $(\eta_{j})_{j=2}^{M}$, where \[\begin{aligned}
\gamma_{1} &= [\bar{\w}_{2\vartheta(i(1)) - 1} o, \bar{\w}_{2\vartheta(i(1))} o],\\
\gamma_{j} &=  [\bar{\w}_{2\vartheta(i(j)) - 2} a_{i(j)} o, \bar{\w}_{2\vartheta(i(j))- 1} o] \quad (j=2, \ldots, M-1), \\
\eta_{j} &= [\bar{\w}_{2\vartheta(i(j)) - 2}o, \bar{\w}_{2\vartheta(i(j)) - 2} a_{i(j)} o] \quad (j = 2, \ldots, M).
\end{aligned}
\]
Moreover, Condition \ref{eqn:addCondPivot} implies that \[
(\bar{\w}_{2\vartheta(i(M)) - 2} a_{i(M)} o, \bar{\w}_{2\vartheta(k+1) - 2} o)_{\bar{\w}_{2\vartheta(i(M)) - 1}o} < K.
\]
In summary, $\{l'=i(1) < \cdots < i(M)\} \subseteq P_{k}(\bar{s})$ works for $\bar{s}$ in Criterion (B) at stage $k+1$, which implies that $P_{k+1}(\bar{s}) \supseteq P_{k}(\bar{s}, a_{k+1}, b_{k+1}) \cap \{1, \ldots, l'\}$, hence the claim.

As a result, we deduce \[\begin{aligned}
&\Prob\Big(\# P_{k+1}(\bar{s}, a_{k+1}, b_{k+1}) < \# P_{k}(s) - 2 \, \Big| \, \bar{s} \in E_{1}(\tilde{s}, \tilde{S}'_{l})\Big) \le 0.1.
\end{aligned}
\]
for each $\tilde{s} \in \mathcal{E}_{k}(s)$ and $(a_{k+1}, b_{k+1}) \in S_{0}^{2}$. Here, for $\tilde{s}$ and $(a_{k+1}, b_{k+1})$ such that $(\tilde{a}_{m}, a_{k+1}, b_{k+1}) \in \mathcal{A}_{1}$, the above probability vanishes. Since \[\begin{aligned}
&\Prob\left[\bigcup \{ E_{1}(\tilde{s}, \tilde{S}_{l}') \times (a_{k+1}, b_{k+1}) : (\tilde{a}_{m}, a_{k+1}, b_{k+1}) \notin \mathcal{A}_{1} \} \, \Big| \, \mathcal{E}_{k}(s) \times S_{0}^{2}  \right] \\
&= \Prob \left[ (\tilde{a}_{m}, a_{k+1}, b_{k+1}) \notin \mathcal{A}_{1}  \, \Big| \, \tilde{S}_{m}'(s) \times S_{0}^{2}\right] \le 0.01,
\end{aligned}
\]
we sum up the conditional probabilities to obtain  \begin{equation}\label{eqn:2ndPivotEst}
\Prob\Big(\# P_{k+1}(\tilde{s}, a_{k+1}, b_{k+1}) < \# P_{k}(s) - 2 \, \Big| \, \tilde{s} \in \mathcal{E}_{k}(s)\Big) \le 0.001.
\end{equation}
We repeat this procedure to cover all $j< \# P_{k}(s)$. The case $j \ge \# P_{k}(s)$ is void.
\end{proof}

\begin{cor}
Conditioned on paths $\mathbf{g} \in G^{n}$ such that $\# \Theta(\mathbf{g}) = N$, $\# P_{N}(\mathbf{g})$ is greater in distribution than the sum of $N$ i.i.d. $X_{i}$, whose distribution is given by \begin{equation}\label{eqn:expRV}
\Prob(X_{i}=j) = \left\{\begin{array}{cc} 9/10 & \textrm{if}\,\, j=1,\\ 9/10^{-j+1} & \textrm{if}\,\, j < 0, \\ 0 & \textrm{otherwise.}\end{array}\right.
\end{equation}
\end{cor}

The RV $X_{i}$ in the above corollary satisfies \[\begin{aligned}
\E\left[X_{i}\right] &= \frac{9}{10} - \frac{9}{10} \left(\frac{1}{10} + \frac{2}{10^{2}} + \ldots\right) = \frac{9}{10} - \frac{1}{9} = \frac{71}{90}, \\
\E\left[1.4^{-X_{i}}\right] &= \frac{5}{7} \cdot \frac{9}{10} + \sum_{j=1}^{\infty} \frac{9}{10} \cdot \left( \frac{7}{50}\right)^{j} = \frac{9}{14} + \frac{9}{10} \cdot \frac{7}{50} + \frac{1}{1-7/50} = \frac{1188}{1505}.
\end{aligned}
\]

\begin{proof}
Lemma \ref{lem:0thCasePivot} and Lemma \ref{lem:pivotCondition} imply the following: for $0\le k <N$ and any $i$, \begin{equation}\label{eqn:probDistbn}
\Prob\left( \# P_{k+1}(\mathbf{g})\ge i + j \, \Big| \, \#P_{k}(\mathbf{g}) = i\right) \ge  \left\{\begin{array}{cc} 1-1/10 & \textrm{if}\,\, j=1,\\ 1 - 1/10^{-j+1} & \textrm{if}\,\, j < 0, \\ 0 & \textrm{otherwise.}\end{array}\right.
\end{equation}
Hence, there exists a nonnegative RV $U_{k}$ such that $\# P_{n+1} - U_{k}$ and $\# P_{k} + X'$ have the same distribution, where $X'$ is an i.i.d. copy of $X_{k+1}$ that is independent from $\# P_{k}$.

For each $1 \le k \le N$, we claim that $\Prob(\#P_{k}(\mathbf{g}) \ge i) \ge \Prob (X_{1} + \cdots + X_{k} \ge i)$ for each $i$.  For $k=1$, we have $\#P_{k-1}(\mathbf{g})= 0$ always and the claim follows from Inequality \ref{eqn:probDistbn}. Given the claim for $k$, we have \[\begin{aligned}
\Prob(\#P_{k+1}\ge i) &\ge \Prob(\#P_{k} + X' \ge i) = \sum_{j} \Prob(\#P_{k}\ge j) \Prob(X' = i - j)\\
 & \ge \sum_{j} \Prob(X_{1} + \cdots + X_{k} \ge j) \Prob(X_{k+1} = i-j) \\
 &= \Prob(X_{1} + \cdots + X_{k} + X_{k+1} \ge i). \qedhere
 \end{aligned}
\]
\end{proof}

Given $\mathbf{g} \in G^{n}$ with \[\begin{aligned}
\Theta(\mathbf{g}) &= \{\vartheta(1) < \ldots < \vartheta(N)\}, \\
P_{N}(\mathbf{g}) &= \{\iota(1) < \ldots < \iota(m)\} \subseteq \{1, \ldots, N\},
\end{aligned}
\] we finally define the {$l$-th pivotal time} of $\mathbf{g}$ by $2\vartheta(\iota(l))- 1$ and the \emph{set of pivotal times} $P_{n}^{\ast}(\mathbf{g})$ by \[
 P_{n}^{\ast}(\mathbf{g}) := \{2\vartheta(i) -1: i \in P_{N}(\mathbf{g})\}.
 \]
We also define $\tilde{S}_{2\vartheta(\iota(l))-1}(\mathbf{g}) := \tilde{S}_{\iota(l)}'(s)$ for $l = 1, \ldots, m$. $\bar{\mathbf{g}} \in G^{n}$ is said to be \emph{pivoted from $\mathbf{g}$} if $g_{j} = \bar{g}_{j}$ unless $j \in P_{n}^{\ast}(\mathbf{g})$, in which case we require $\bar{g}_{j} \in \tilde{S}_{j}(\mathbf{g})$. 

Lemma \ref{lem:extremal} and Corollary \ref{cor:induction} imply that \begin{equation}\label{eqn:ineq1}
(\w_{i} o, \w_{k})_{\w_{j}o} < F(K, \epsilon) < K'/4000
\end{equation}
for $i, j, k \in P_{n}^{\ast}(\mathbf{g}) \cup \{0, n\}$ such that $i \le j \le k$. Moreover, for such $i, j, k \in P_{n}^{\ast}(\mathbf{g}) \cup \{0\}$ such that $i < j \le k$, we have  \begin{equation}\label{eqn:ineq2}\begin{aligned}
1.999K' \le d(\w_{k-1} o, \w_{k} o) &\le 2K', \\ (\w_{i }o, \w_{k} o)_{\w_{i-1} o} < F(K, \epsilon) &< K'/4000.
\end{aligned}
\end{equation}
The first inequality is due to the fact that $g_{k} \in S_{1}^{(2)}\cup S_{1}^{(-2)}$; the second inequality follows from Lemma \ref{lem:intermediate} and Corollary \ref{cor:induction}.

\section{Pivoting and translation lengths}\label{section:trLength}

We will now define another equivalence relation on paths with sufficiently many pivots. Let us fix $\mathbf{g} \in G^{n}$ with $P_{n}^{\ast}(g) = \{i(1) < \ldots < i(m)\}$, where $m := \#P_{n}^{\ast}(\mathbf{g})$ satisfies $n/5 \le m \le n / 2$. For convenience, let also $i(0) = 0$. We now define quantities \[\begin{aligned}
D_{f}(\mathbf{g}) &= \sum_{l=1}^{\lfloor n/12 \rfloor} [d(\w_{i(l-1)}o, \w_{i(l)-1}o) + 2K'],\\
D_{b}(\mathbf{g}) &= \sum_{l=m-\lfloor n/12 \rfloor+1}^{m}[d(\w_{i(l-1)}o, \w_{i(l) - 1}o) + 2K']+ d(\w_{i(m)}o, \w_{n}o),\\
D_{t}(\mathbf{g}) &=  \sum_{l=1}^{m} \,[d(\w_{i(l-1)}o, \w_{i(l) - 1}o) + 2K']+ d(\w_{i(m)}o, \w_{n}o).
\end{aligned}
\]

Note the inequality \begin{equation}\label{eqn:DfControl}\begin{aligned}
D_{f}(\mathbf{g}) &\ge  \sum_{l=1}^{\lfloor n/12 \rfloor} [d(\w_{i(l-1)}o, \w_{i(l) - 1}o) + d(\w_{i(l) - 1} o, \w_{i(l)} o)]\\
& \ge d(o, \w_{i(k) - 1}),\, d(o, \w_{i(k)})
\end{aligned}
\end{equation}
for $k=1, \ldots, \lfloor n/12 \rfloor$. Similarly, $D_{t}$ dominates $d(o, \w_{n} o)$. Moreover, due to Inequality \ref{eqn:ineq1} and \ref{eqn:ineq2}, we have $|d(o, \w_{n} o) - D_{t}| < 2F(K, \epsilon) n\le K'n/1000$. We also observe that at least one of $D_{f}, D_{b}$ is smaller than $D_{t}/2 - K' n/20$; indeed, $D_{t} - D_{f} - D_{b}$ is the sum of at least $n/20$ terms of the form $d(\w_{i(l - 1)}o, \w_{i(l)-1} o)+ 2K'  \ge 2K'$. 

If $D_{f} \le D_{b}$, then we allow pivoting at the first $\lfloor n/12 \rfloor$ pivotal times. Otherwise, we allow pivoting at the last $\lfloor n/12 \rfloor$ pivotal times. Since $D_{f}$, $D_{b}$, $D_{t}$, and the set of pivotal times are invariant under pivoting, this rule partitions $\{\mathbf{g} \in G^{n} : \# P_{n}^{\ast}(\mathbf{g}) \ge n/5\}$ into equivalence classes $\mathcal{F}(\mathbf{g})$'s.

We are now ready to prove the core lemma for Theorem \ref{thm:generic}.

\begin{lem}\label{lem:trLength}
Let $n>25$ and suppose that $\mathbf{g} \in G^{n}$ satisfies \[
\# P_{n}^{\ast}(\mathbf{g}) \ge n/5, \quad D_{f}(\mathbf{g}) \le D_{b}(\mathbf{g}).
\] Let also $1\le k < k' \le \lfloor n/12 \rfloor $ and $\tilde{g}_{i(l)} \in \tilde{S}_{i(l)}(\mathbf{g})$ for $l=1, \ldots, k-1, k'+1, \ldots, \lfloor n/12 \rfloor$.

Then there exist $A \subseteq \tilde{S}_{i(k)}(\mathbf{g})$ and $A' \subseteq \tilde{S}_{i(k')}(\mathbf{g})$, each of cardinality at most 2, such that the following holds: for any $\bar{\mathbf{g}} \in \mathcal{F}(\mathbf{g})$ such that $\bar{g}_{i(l)} = \tilde{g}_{l}$ for $l=1, \ldots, k-1, k'+1, \ldots, \lfloor n/12 \rfloor$ and $\bar{g}_{i(k)} \notin A^{(2)}$, $\bar{g}_{i(k')} \notin A'^{(2)}$, we have $\tau(\bar{\w}_{n}) \ge K'n/12$.
\end{lem}

\begin{proof}
By the assumption $\bar{\mathbf{g}} \in \mathcal{F}(\mathbf{g})$, $\bar{\mathbf{g}}$ can differ only at step $i(1)$, $\ldots$, $i(\lfloor n/12 \rfloor)$. Hence, for $\bar{\mathbf{g}} \in \mathcal{F}(\mathbf{g})$ such that $\bar{g}_{i(l)} =\tilde{g}_{i(l)}$ for $l=1, \ldots, k-1, k'+1, \ldots, \lfloor n/12 \rfloor$, the isometry \[
v := (\bar{\w}_{i(k')})^{-1} \bar{\w}_{n} \bar{\w}_{i(k) - 1} = \bar{g}_{i(k')+ 1} \cdots \bar{g}_{n} \cdot \bar{g}_{1} \cdots \bar{g}_{i(k)-1}
\]
is uniform. We define \[\begin{aligned}
A := \{g \in \tilde{S}_{k}(\mathbf{g}) : (v^{-1} o, g^{2} o)_{o} \ge K\}, \\
A':= \{g \in \tilde{S}_{k'}(\mathbf{g}) : (g^{-1} o, vo)_{o} \ge K\}.
\end{aligned}
\]
Since $\tilde{S}_{k}(\mathbf{g}) \subseteq S_{0} = S_{1} \cup S_{1}^{(-1)}$, Property (5), (6) of Schottky sets imply that $\#A \le 2$. Similarly we have $\#A' \le 2$.

Let us now fix $h \in \tilde{S}_{k}(\mathbf{g}) \setminus A$, $h' \in \tilde{S}_{k'} (\mathbf{g}) \setminus A'$ and consider $\bar{\mathbf{g}} \in \mathcal{F}(\mathbf{g})$ such that $\bar{g}_{\rho(l)} = \tilde{g}_{l}$ for $l=1, \ldots, k-1, k' + 1, \ldots, \lfloor n/12 \rfloor$ and $\bar{g}_{\rho(k)}=h^{2}$, $\bar{g}_{\rho(k')}= h'^{2}$. 

Since $i(k), i(k') \in P_{n}^{\ast}(\mathbf{g}) = P_{n}^{\ast}(\bar{\mathbf{g}})$ and $h'^{2} = \bar{g}_{i(k')}$, Lemma \ref{lem:intermediate2} gives Schottky segments $(\gamma_{l})_{l=1}^{M-1}$, $(\eta_{l})_{l=2}^{M}$ such that $[\bar{\w}_{i(k)} o, \bar{\w}_{i(k')- 1} h' o]$ is fully $D(K, \epsilon)$-marked with $(\gamma_{l})$, $(\eta_{l})$, where $\gamma_{1} = [\bar{\w}_{i(k)} o, \bar{\w}_{i(k)+1} o]$, $\eta_{M} = [\bar{\w}_{i(k') - 1} o, \bar{\w}_{i(k') - 1} h' o]$. 

Next, Inequality \ref{eqn:DfControl} implies that \[\begin{aligned}
d(o, vo) &\ge d(o, \bar{\w}_{n} o) - d(o, \bar{\w}_{i(k')} o) - d(o, \bar{\w}_{i(k) -1} o) \\
&\ge (D_{t} - K'n/1000) - 2D_{f} \\
&\ge K'n/12\ge 2K' + 3D(K, \epsilon).
\end{aligned}
\]
Since we also have $(h'^{-1} o, vo)_{o} \le K$ and $(o, vh^{2} o)_{vo} \le K$, Fact \ref{lem:farSegment} implies that $[h'^{-1} o, v h^{2} o]$ is $D(K, \epsilon)$-witnessed by $([h'^{-1} o, o], [vo, v h^{2} o])$. By applying isometry $\bar{\w}_{n}^{i-1} \bar{\w}_{i(k')}$, we deduce that $[\bar{\w}_{n}^{i-1}\bar{\w}_{i(k') - 1} h' o, \bar{\w}_{n}^{i}  \bar{\w}_{i(k)}o]$ is $D(K, \epsilon)$-witnessed by Schottky segments \[
[\bar{\w}_{n}^{i-1}\bar{\w}_{i(k')- 1} h' o, \bar{\w}_{n}^{i-1}\bar{\w}_{i(k')} o], \quad [\bar{\w}_{n}^{i} \bar{\w}_{i(k) - 1} o, \bar{\w}_{n}^{i}  \bar{\w}_{i(k)}o].
\]

We now claim that $[\bar{\w}_{i(k)} o, \bar{\w}_{n}^{i} \bar{\w}_{i(k)} o]$ is fully $D(K, \epsilon)$-witnessed by
 \[\begin{aligned}
(\gamma_{1}, \ldots, \gamma_{M-1}, [\bar{\w}_{i(k') - 1} h' o, \bar{\w}_{i(k')} o], \quad \bar{\w}_{n}\gamma_{1}, \ldots, \bar{\w}_{n}\gamma_{M-1},[\bar{\w}_{n}\bar{\w}_{i(k') - 1} h' o, \bar{\w}_{n}\bar{\w}_{i(k')} o], \,\, \ldots), \\
(\eta_{2}, \ldots, \eta_{M},[\bar{\w}_{n} \bar{\w}_{i(k)-1} o, \bar{\w}_{n} \bar{\w}_{i(k)}o],  \quad \bar{\w}_{n}\eta_{2}, \ldots, \bar{\w}_{n}\eta_{M}, [\bar{\w}_{n}^{2} \bar{\w}_{i(k)-1} o, \bar{\w}_{n}^{2}  \bar{\w}_{i(k)}o],  \quad\quad\, \,\,\ldots).
\end{aligned}
\]
This claim will follow from Lemma \ref{lem:markedConcat} once we check \[\begin{aligned}
(\bar{\w}_{i(k') - 1}o, \bar{\w}_{i(k')} o)_{\bar{\w}_{i(k')- 1} h' o} &= (h'^{-1} o, h' o)_{o} < D(K, \epsilon),\\
(\bar{\w}_{i(k)-1} o, \bar{\w}_{i(k)+1} o)_{\bar{\w}_{i(k)} o} &= (h^{-2} o, g_{i(k)+1} o)_{o} < D(K, \epsilon).
\end{aligned}
\]
The first item follows from Property (7) of Schottky sets, and the second item follows from $h \in \tilde{S}_{k}(\mathbf{g})$; hence the claim. In particular, Corollary \ref{cor:induction} implies $(\bar{\w}_{i(k)} o,  \bar{\w}_{n}^{i} \bar{\w}_{i(k)} o)_{\bar{\w}_{n}^{i-1} \bar{\w}_{i(k)} o} < F(K, \epsilon)$ for each $i \ge 1$ and \[\begin{aligned}
\frac{1}{i} d\left(\bar{\w}_{i(k)} o, \,\, \bar{\w}_{n}^{i} \bar{\w}_{i(k)} o\right) &\ge d(\bar{\w}_{i(k)} o,  \bar{\w}_{n} \bar{\w}_{i(k)} o) - F(K, \epsilon) \\
& \ge [d(o, \w_{n} o) - d(o, \bar{\w}_{i(k)} o) - d(\bar{\w}_{n}o, \bar{\w}_{n}\bar{\w}_{i(k)} o)] - K'n/1000\\
&\ge D_{t} - 2D_{f} - K'n/500 \ge K'n/12.
\end{aligned}
\]
By sending $i \rightarrow \infty$, we deduce that $\tau(\bar{\w}_{n}) \ge K'n/12$.
\end{proof}

%
%

\section{Proof of Theorem \ref{thm:generic}}\label{section:proof}

Let us now prove Theorem \ref{thm:generic}.

\begin{proof}
Let $S' \subseteq G$ be the given finite set. By using Lemma \ref{lem:Schottky}, we take a $(K, K', \epsilon)$-Schottky subset $S_{1}$ of $G$ such that $K'>2L(K ,\epsilon) + 5000 F(K, \epsilon)$ and a finite symmetric generating set $S \supseteq S'$ such that $e \in S$ and $S$ is nicely populated by $S_{1}^{(2)} \cup S_{1}^{(-2)}$.

As before, we consider the random walk on $G$ generated by the uniform measure $\mu_{S}$ on $S$. We first claim \[
\Prob(\#P_{n}^{\ast} \le n/5) \le \frac{n}{3} \cdot 0.9^{n} + 0.9886^{n}.
\]
for large $n$. The first term of the RHS is for the event of trajectories $\mathbf{g}$ with $\#\Theta(\mathbf{g}) \le n/3$, 
whose probability is at most \[
\sum_{i=0}^{\lfloor n/3 \rfloor} \binom{\lfloor n/2 \rfloor}{i} \cdot 0.99^{i} \cdot 0.01^{\lfloor n/2 \rfloor - i} \le \sum_{i=0}^{\lfloor n/3 \rfloor} \binom{\lfloor n/2 \rfloor}{\lfloor n/3 \rfloor} \cdot 0.99^{\lfloor n/3 \rfloor } \cdot 0.01^{n/6 - 1} \le \frac{n}{3}  \cdot0.9^{n}.
\]
Here the final inequality is deduced from the fact that \[
\binom{3(m+1)}{2(m+1)} \cdot 0.01^{m+1} = \binom{3m}{2m} \cdot \frac{(3m+3)(3m+2)(3m+1)}{(m+1)(2m+1)(2m+2)} \cdot 0.01^{m+1} \le \binom{3m}{2m} \cdot 0.01^{m} \cdot 0.07
\]
for sufficiently large $m$, and that $0.07^{1/6} < 0.9$.

The second term is an estimation for the sum of $m$ i.i.d. RVs $X_{i}$ of the distribution in Equation \ref{eqn:expRV}. Recall that $X_{i}$ is an RV with exponential tail and $\E[X_{i}] = 71/90$. Hence for $\lambda< 71/90$, the theory of large deviation says that $\Prob\left( \sum_{i=1}^{m} X_{i} < \lambda m\right) \le e^{-\kappa(\lambda) m}$ for some $\kappa(\lambda) > 0$. The easiest way to show this (for suitable $\lambda$) is to take an intermediate base $\lambda< \lambda_{0} < \E[X_{i}]$ and apply Markov's inequality to the RV $\lambda_{0}^{\sum_{i} X_{i}}$. Indeed, Markov's inequality tells us that \[
\Prob\left(\sum_{i=1}^{m} X_{i} < n/5\right) \cdot 1.4^{-n/5} \le \E\left[1.4^{-\sum_{i=1}^{m} X_{i}}\right] =\prod_{i=1}^{m} \E\left[1.4^{-X_{i}}\right]= \left(\frac{1188}{1505}\right)^{m}
\]
and the desired estimate follows for $m \ge n/3$.

We now consider an equivalence class $\mathcal{F}(\mathbf{g})$ of $\mathbf{g} \in G^{n}$ such that $\#P_{n}(\mathbf{g}) \ge n/5$ and $D_{f}(\mathbf{g}) \le D_{b}(\mathbf{g})$. For $h_{i} \in \tilde{S}_{i}(\mathbf{g})$ and $k=1, \ldots, \lfloor n/24 \rfloor$, Lemma \ref{lem:trLength} gives the sets
\[\begin{aligned}
&A_{k} (\mathbf{g}, \{ h_{l}, h_{\lfloor n/12 \rfloor - l}\}_{l=1}^{k-1}) \subseteq \tilde{S}_{i(k)}(\mathbf{g}), \\
&A_{k}' (\mathbf{g}, \{ h_{l}, h_{\lfloor n/12 \rfloor - l}\}_{l=1}^{k-1}) \subseteq \tilde{S}_{i(\lfloor n/12 \rfloor - k)}(\mathbf{g})
\end{aligned}
\] with cardinality at most 2, such that the following holds: $\bar{g} \in \mathcal{F}(\mathbf{g})$ satisfies $\tau(\bar{\w}_{n}) < K'n/12$ only if $\bar{g}_{\rho(k)} \in A_{k}(\mathbf{g},  \{ \bar{g}_{i(l)}, \bar{g}_{i(\lfloor n/12 \rfloor - l)}\}_{l=1}^{k-1})$ or $\bar{g}_{i(\lfloor n/12 \rfloor - k)} \in A_{k}'(\mathbf{g},  \{ \bar{g}_{i(l)}, \bar{g}_{u(\lfloor n/12 \rfloor - l)}\}_{l=1}^{k-1})$ holds for each $k=1, \ldots, \lfloor n/24 \rfloor$. This implies that
\[\begin{aligned}
&\Prob\left( \tau(\bar{\w}_{n}) \le K'n /12 \, \Big| \, \bar{\mathbf{g}} \in \mathcal{F}(\mathbf{g}) \right) \\
&\le \prod_{k=1}^{\lfloor n/24 \rfloor}  \frac{(\# \tilde{S}_{i(k)})(\# \tilde{S}_{i(\lfloor n/12 \rfloor - k)})- (\# \tilde{S}_{i(k)}-2)(\# \tilde{S}_{i(\lfloor n/12 \rfloor - k)}-2)}{(\# \tilde{S}_{i(k)})(\# \tilde{S}_{i(\lfloor n/20 \rfloor - k)})}
\\&\le \prod_{k=1}^{\lfloor n/24 \rfloor} \left[ \frac{2}{\# \tilde{S}_{i(k)}}+ \frac{2}{\# \tilde{S}_{i(\lfloor n/12 \rfloor - k)}} \right] \\
&\le \prod_{k=1}^{\lfloor n/24 \rfloor} \left[ \frac{2}{0.99 \#S} + \frac{2}{0.99\#S}\right]= (0.2475 \#S)^{-\lfloor n/24\rfloor}.
\end{aligned}
\]
A similar argument leads to the same conclusion for $\mathcal{F}(\mathbf{g})$'s where $D_{f}(\mathbf{g}) > D_{b}( \mathbf{g})$. Since these $\mathcal{F}(\mathbf{g})$'s partition $\{\# P_{n}  \ge n/5\}$, we conclude that the number of $n$-step trajectories $\w$ such that $\tau(\w_{n}) \le K'n/12$ is bounded by \[
(\#S)^{n} \cdot [0.91^{n} + 0.9886^{n} + (0.2475\#S)^{-n/24}] \le 0.999^{n} \cdot (0.99 \#S)^{n}
\]
for sufficiently large $n$. Since any mapping class in $B_{S}(n)$ is obtained from an $n$-step trajectory, we conclude that the number of mapping classes in $B_{S}(n)$ with translation length less than $K'n/12$ is bounded by $0.999^{n} \cdot (0.99\#S)^{n}$. 

Meanwhile, the set \[
S_{Schottky} = \{(a_{1}, \ldots, a_{n}): a_{i} \in S_{1} \cup S_{1}^{-1}, a_{i} \neq a_{i+1}^{-1}\}
\] is composed of at least $(\#S_{0} - 1)^{n} \ge (0.99\#S)^{n}$ sequences. We claim that if $(a_{1}, \ldots, a_{n}), (b_{1}, \ldots, b_{n})$ are distinct sequences in $S_{Schottky}$, then $a_{1}^{2} \cdots a_{n}^{2}$, $b_{1}^{2} \cdots b_{2}^{2}$ are distinct elements in $B_{S}(n)$. Indeed, the sequence \[
(a_{n}^{-1}, a_{n}^{-1} \cdots, a_{1}^{-1} , a_{1}^{-1}, b_{1}, b_{1}, \ldots, b_{n}, b_{n})
\] will not completely cancel out and their product will not become an identity by Lemma \ref{lem:almostInj}. Hence, we have at least $(0.99\#S)^{n}$ distinct elements in $B_{S}(n)$. We thus finally have \[
\frac{\#\{g \in B_{S}(n) : \tau_{X}(g) \le K'n/12\}}{\#B_{S}(n)} \le 0.999^{n}
\]
for large $n$, finishing the proof.
\end{proof}

\appendix

\section{The proof of Claim \ref{claim:prop4.2}}\label{section:remark}

In this section, we prove Claim \ref{claim:prop4.2} in the proof of Lemma \ref{prop:Schottky}. We first recall the following lemma:

\begin{fact}[{\cite[Lemma 3.12]{choi2021clt}}]\label{lem:passerBy}
For each $F, \epsilon > 0$, there exists $H, L > F$ that satisfies the following condition. If $x, y, z, p_{1}, p_{2}$ in $X$ satisfy that: \begin{enumerate}
\item $[p_{1}, p_{2}]$ is $\epsilon$-thick and longer than $L$,
\item $[x, y]$ is $F$-witnessed by $[p_{1}, p_{2}]$, and
\item $(x, z)_{y} \ge d(p_{1}, y)-F$,
\end{enumerate}
then $[z, y]$ is $H$-witnessed by $[p_{1}, p_{2}]$.
\end{fact}

Recall that we have fixed $o \in X$ and independent loxodromics $a, b \in G$. By Lemma 4.3 and 4.4 of \cite{choi2021clt}, there exists $\epsilon_{0}, C_{0}>0$ such that the following hold: \begin{enumerate}
\item $[o, a^{i} o]$, $[o, b^{i} o]$ are $\epsilon_{0}$-thick for all $i \in \Z$, and
\item $(\phi^{i} o, \psi^{j} o)_{o} < C_{0}$ for all $i, j > 0$ and $\phi, \psi \in \{a, b, a^{-1}, b^{-1}\}$ such that $\phi \neq \psi$.
\end{enumerate}
 
We then define: \begin{itemize}
\item $D_{0} = D(C = C_{0}, \epsilon_{0})$ as in Fact \ref{lem:1segment};
\item $E_{0} = E(D = D_{0}, \epsilon_{0})$, $L_{0} =L(D = D_{0}, \epsilon_{0})$ as in Fact \ref{lem:concat};
\item $F_{0} = F(E = E_{0}, \epsilon_{0})$, $L_{1} = L(E = E_{0}, \epsilon_{0})$ as in Fact \ref{lem:concatUlt};
\item $H_{0} = H(F = F_{0}, \epsilon_{0})$, $L_{2} = L(F = F_{0}, \epsilon_{0})$ as in Fact \ref{lem:passerBy}; 
\item $H_{1} = H(F = H_{0}, \epsilon_{0})$ $L_{3} = L(F = H_{0}, \epsilon_{0})$ as in Fact \ref{lem:passerBy}; 
\item $F_{1} = 2F_{0} + H_{1} +\delta + 1$;
\item $F_{2} = F(E = H_{0}, \epsilon_{0})$ and $L_{4} = L(E = H_{0}, \epsilon_{0})$ in Fact \ref{lem:concatUlt}; 
\item $L_{5} = \max( L_{0}, L_{1}, L_{2}, L_{3}, L_{4}, 2F_{0}+ F_{1} + 2F_{2})$. 
\end{itemize}

There exists $N_{0}$ such that $d(o, \phi^{N} o) > L_{5}$ for all $\phi \in \{a, b, a^{-1}, b^{-1}\}$ and $N> N_{0}$. Let us fix $N> N_{0}$. We now consider a sequence $\{\phi_{i}\}$ in $\{a, b, a^{-1}, b^{-1}\}$ such that $\phi_{i} \neq \phi_{i+1}^{-1}$. Then we observe the following:\begin{enumerate}
\item $(\phi_{i}^{-N} o, \phi_{i+1}^{N})_{o} < C_{0}$ by the assumption $\phi_{i} \neq \phi_{i+1}^{-1}$.
\item $[o, \phi_{i}^{N} \phi_{i+1}^{N} o]$ is $D_{0}$-witnessed by $[o, \phi_{i}^{N} o]$ and $[\phi^{N} o, \phi^{N}\phi_{i+1}^{N} o]$ by Fact \ref{lem:1segment}.
\item $[o, \phi_{i}^{N} o]$ are $\epsilon$-thick and longer than $L_{5}$.
\end{enumerate}
Then as in Corollary \ref{cor:induction},  Fact \ref{lem:concat} and Fact \ref{lem:concatUlt} imply that $[o, \phi_{1}^{2N} \cdots \phi_{n}^{2N} o]$ is $F_{0}$-witnessed by $\epsilon_{0}$-thick segments\[
[o, \phi_{1}^{N}o], [\phi_{1}^{N}o, \phi_{1}^{N}\phi_{2}^{N} o], \ldots, [\phi_{1}^{N} \cdots \phi_{n-1}^{N}o, \phi_{1}^{N} \cdots \phi_{n}^{N}o].
\] Consequently, $[o, \phi_{1}^{N} \cdots \phi_{n}^{N} o]$ is $\epsilon$-thick for $\epsilon = \epsilon_{0}e^{-8F_{0}}$. It is also clear that \[
(o, \phi_{1}^{N} \cdots \phi_{n}^{N} o)_{\phi_{1}^{N} \cdots \phi_{m}^{N} o} < F_{0}\quad (0 \le m \le n)
\]
 and \begin{equation}\label{eqn:nonDecrease}
 d(o, \phi_{1}^{N} \cdots \phi_{n}^{N} o) \ge d(o, \phi_{1}^{N} \cdots \phi_{n-1}^{N} o) + L_{1} - 2F_{0} \ge d(o, \phi_{1}^{N} \cdots \phi_{n-1}^{N} o) + F_{1}.
 \end{equation}

We now define \[\begin{aligned}
S_{n, N} &:= \{g_{1}, \ldots, g_{2^{n}}\} = \{\phi_{1}^{2N} \cdots \phi_{n}^{2N} : \phi_{i} \in \{a, b\}\},\\
V(g_{i}^{\pm}) &:= \{x \in X : (x, g_{i}^{\pm 2} o)_{o} \ge d(o, g_{i}^{\pm 1} o)-F_{1}\}, \\
V'(g_{i}^{\pm}) &:=  \{x \in X : (x, g_{i}^{\pm 2} o)_{o} \ge d(o, g_{i}^{\pm 1} o)\}.
\end{aligned}\]
Our first claim is that $V(g_{1}^{+}), \ldots, V(g_{2^{n}}^{+}), V(g_{1}^{-}), \ldots, V(g_{2^{n}}^{-})$ are all disjoint. To show this, let $(\phi_{i})_{i=1}^{n}, (\psi_{i})_{i =1}^{n}$ be distinct sequences in $\{a, b\}^{n} \cup \{a^{-1}, b^{-1} \}^{n}$ and $\Phi = \phi_{1}^{2N} \cdots \phi_{n}^{2N}$, $\Psi = \psi_{1}^{2N} \cdots \psi_{n}^{2N}$. Let $t = \min \{ 1 \le i \le 10 : \phi_{i} \neq \psi_{i}\}$ and $w = \phi_{1}^{2N} \cdots \phi_{t-1}^{2N}$. Now suppose that a point $x \in X$ belongs to both $V(\Phi)$ and $V(\Psi)$. First, $x \in V(\Phi)$ implies  \[
(x, \Phi^{2} o)_{o} \ge d(o, \Phi o) - F_{1} \ge d(o, \phi_{1}^{2N} \cdots \phi_{t-1}^{2N} \phi_{t}^{N} o) = d(o, w \phi_{t}^{N} o)
\]
by Inequality \ref{eqn:nonDecrease}. Since $[o, \Phi^{2} o]$ is $F_{0}$-witnessed by $[wo, w\phi_{t}^{N} o]$, Fact \ref{lem:passerBy} asserts that $[o, x]$ is $H_{0}$-witnessed by $[wo, w\phi_{t}^{N} o]$. By a similar reason, we have $(x, \Psi^{2} o)_{o} \ge d(o, w\psi_{t}^{N} o)$ and that $[o, x]$ is $H_{0}$-witnessed by $[w o, w\psi_{t}^{N} o]$. Since $(\phi_{t} o, \psi_{t} o)_{o} < C_{0} < H_{0}$, Fact \ref{lem:concatUlt} implies that $[x, x]$ is $F_{2}$-witnessed by $[wo, w\psi_{t} o]$, whose length is at least $L_{5} >2F_{2}$: such $x$ does not exist.

The next claim is that if $x \notin V(g_{i}^{-})$, then $g_{i}^{2} x \in V'(g_{i})$. Indeed, we know that $(o, g_{i}^{2} o)_{g_{i} o} \le F_{0} \le F_{1}/2$ and  \[\begin{aligned}
(g_{i}^{2}x, g_{i}^{2} o)_{o} &= (x, o)_{g_{i}^{-2} o} = d(o, g_{i}^{-2} o) - (x, g_{i}^{-2} o)_{o}\\
& \ge d(o, g_{i}^{2} o) - d(o, g_{i} o) +F_{1} \ge d(o, g_{i}^{} o).
\end{aligned}
\]

Since $V'(g_{i}) \subseteq V(g_{i})$ and $V(g_{i}) \cap V(g_{i}^{-}) = \emptyset$, we can iterate this to deduce $g_{i}^{2k} x \in V'(g_{i})$ for $k > 0$. Similarly, if $x \notin V(g_{i})$, then $g_{i}^{-2k} x \in V'(g_{i}^{-})$ for $k > 0$. 

Now let $x, y\in X$. Since $\{V(g_{i}^{+}), V(g_{i}^{-})\}$ are disjoint, $y \in V(g_{i}^{-})$ for at most one $g_{i} \in S_{n, N}$ and $x \in V(g_{j}^{+})$ for at most one $g_{j} \in S_{n, N}$. Suppose $s=\phi_{1}^{2N} \cdots \phi_{n}^{2N} \in S_{n, N}$ is neither of them, and let $k>0$. We then have \[
(x, s^{2} o)_{o} < d(o, so) - F_{1}, \quad (s^{2k}y, s^{2} o)_{o} \ge d(o, s o).
\] As $[o, s^{2} o]$ is $F_{0}$-witnessed by $[s \phi_{n}^{-N} o, so]$, Fact \ref{lem:passerBy} implies that $[o, s^{2k}y]$ is $H_{0}$-witnessed by $[s \phi_{n}^{-N} o, so]$. Now if we suppose that $(x, s^{2k} y)_{o} > d(o, so)$, then $[o, x]$ is also $H_{1}$-witnessed by $[s \phi_{n}^{-N} o, so]$, again by Fact \ref{lem:passerBy}. Meanwhile: \begin{itemize}
\item if $X$ is a $\delta$-hyperbolic space, we deduce that \[\begin{aligned}
(x, s^{2} o)_{o}& \ge \min \{ (x, so)_{o}, (so, s^{2} o)_{o} \} - \delta \\
&= d(o, so) - \max\{ (x, o)_{so}, (s^{2}o, o)_{so}\} - \delta \\
&\ge d(o, so) - (H_{1} + F_{0} + \delta) \ge d(o, so) - F_{1}.
\end{aligned}
\]
\item if $X = \T(\Sigma)$, then we deduce that \[\begin{aligned}
(x, s^{2} o)_{o}& \ge d(o, so) - d(so, [o, x]) - d(so, [o, s^{2} o]) \\
&\ge d(o, so) - (H_{1} + F_{0}) \ge d(o, so) - F_{1}.
\end{aligned}
\]
\end{itemize}
In either case we obtain a contradiction. Hence, $(x, s^{2k} y)_{o} \le d(o, so)$.

Similarly, if $s \neq g_{i}$ such that $y \in V(g_{i}^{+})$ and $s \neq g_{j}$ such that $x \in V(g_{j}^{-})$, then $(x, s^{-2k} y)_{o} \le d(o, s^{-1} o)$ for all $k > 0$. Finally, note that $y = o$ cannot belong to any of $V(g_{j}^{\pm})$ since $d(o, g_{j}^{\pm 1}o) \ge L_{5}> F_{1}$ for any $g_{j} \in S_{n, N}$. This settles the desired claim.

\section{Sketch of the proof of Proposition \ref{prop:genericGekht}}\label{section:appendix}

We borrow the definitions and notations in \cite{gekhtman2018counting}.

In \cite{gekhtman2018counting}, the authors consider the automatic structure of $G$, a directed graph $\Gamma$ that records exactly one geodesic between $e$ and $g$ for each $g \in G$. Hence, the vertex set of its universal cover $\tilde{\Gamma}$ and $G$ has 1-1 correspondence. Let $LG$ be the set of vertices with large growth. 
For $g \in G$ and $0<\epsilon<1$, we denote by $\hat{g}_{\epsilon}$ the element along the path from $e$ to $g$ at distance $\epsilon n$ from $e$. Then for any $0<\epsilon<1$, the ratio\begin{equation}\label{eqn:gekht1}
\frac{\#\{g \in \partial B_{S}(n) : \hat{g}_{\epsilon} \notin LN\}}{\# \partial B_{S}(n)}
\end{equation}
decays exponentially (cf. \cite[Proposition 2.5]{gekhtman2018counting}).

The authors then construct a Markov chain on $\Gamma$ whose $n$-step distribution is denoted by $\Prob^{n}$. There exists $c>1$ such that for any $A \subseteq \tilde{\Gamma}$, the proportion of $A\cap LG$ in $\partial B_{S}(n)$ is at least $(1/c)\Prob^{n}(A)$ and at most $c\Prob^{n}(A)$. 

We now denote by $\mathcal{R}$ the set of recurrent vertices. For $v \in \mathcal{R}$, the loop semigroup $\Gamma_{v}$ associated to $v$ is nonelementary, i.e., there exist independent loxodromics $a_{v}, b_{v} \in \Gamma_{v}$ \cite[Corollary 6.11]{gekhtman2018counting}. Let us now condition on the paths growing from $v$. Let $n(k, v, \w)$ be the $k$-th return time to $v$ and $T_{v} = \E_{v} n(1, v, \w)$. Then for each $\epsilon> 0$, \begin{equation}\label{eqn:gekht2}
\Prob_{v}\left\{ \left|\frac{n(k, v, \w)}{k} -T_{v} \right| >\epsilon \right\}
\end{equation}
is exponentially decaying as $k \rightarrow \infty$ (cf. \cite[Lemma 6.13]{gekhtman2018counting}). For each $n$, we also define the last return time $\tilde{n}(\w) = \max\left( \{ n(k, v, \w):k \in \N\} \cap \{1, \cdots, n\}\right)$ to $v$. Then for each $\epsilon > 0$, \begin{equation}\label{eqn:gekht3}
\Prob^{n}_{v} \left\{ \frac{n-\tilde{n}(\w)}{n} > \epsilon \right\}
\end{equation}
decays exponentially.

We now strengthen Theorem 6.14 of \cite{gekhtman2018counting}. For each $\epsilon > 0$, \begin{equation}\label{eqn:gekht4}
\Prob_{v}\left\{ \left|\frac{d(\w_{n(k, v)}o, o)}{k} - l_{v} \right|> \epsilon\right\}
\end{equation}
decays exponentially as $k \rightarrow \infty$, since $\mu_{v}$ actually has finite exponential moment. The proof for deviations from above can be found in \cite{boulanger2020large}. For deviations from below, \cite{boulanger2020large} and \cite{gouezel2021exp} deals with the case that $X$ is Gromov hyperbolic. When $X$ is the Teichm{\"u}ller space, one can use Choi's modification of Gou{\"e}zel's construction in \cite{choi2021clt}. Now together with the control on quantities \ref{eqn:gekht2} and \ref{eqn:gekht3}, we obtain the following: for any $\epsilon>0$, \begin{equation}\label{eqn:gekht5}
\Prob_{v}^{n}\left\{ \left| \frac{d(\w_{n} o, o)}{n} - \frac{l_{v}}{T_{v}}\right| >\epsilon\right\}
\end{equation}
decays exponentially. Now the proof of \cite[Theorem 6.14]{gekhtman2018counting} shows that $l_{v}/T_{v}$ is uniform for all $v \in \mathcal{R}$, which we denote by $\lambda$. Since the arrival time at $\mathcal{R}$ (beginning at $e$) also has finite exponential moment, we conclude that \begin{equation}\label{eqn:gekht6}
\Prob_{e}^{n} \left\{ \left| \frac{d(\w_{n} o, o)}{n} -\lambda\right| >\epsilon\right\}
\end{equation}
decays exponentially. By combining this with the decay of \ref{eqn:gekht1}, we deduce that \[
\frac{\# \{ g \in \partial B_{S}(n) : |\frac{d(o, go)}{n}- \lambda|>\epsilon \}}{\# \partial B_{S}(n)}
\]
decays exponentially (cf. \cite[Theorem 7.3]{gekhtman2018counting}).

We now need to discuss translation lengths instead of displacements. For each recurrent component $C$ of $\Gamma$, we pick $v = v_{C} \in C$ and take a Schottky set as a subset of $\{w = g_{1} \cdots g_{n} : g_{i} = a_{v}\,\,\textrm{or}\,\, b_{v}\}$. We now consider the loop random walk generated by $\mu_{v}$: recall the decomposition of the random walk into usual steps and ``Schottky steps" for the pivot construction in \cite{gouezel2021exp} and \cite{choi2021clt}. 

Until step $T_{v} n/4$ in the loop random walk, we have $Kn$ slots for Schottky steps for some $K>0$ outside an event of exponentially decaying probability. Moreover, $T_{v}n/4$ steps in the loop random walk occur before step $n/2$ in the Markov process based at $v_{C}$, outside an event of exponentially decaying probability (quantity \ref{eqn:gekht2}). Finally, the Markov process beginning from $e$ arrives at $\{v_{C} : C\,\, \textrm{recurrent}\}$ within step $n/2$ outside an event of exponentially decaying probability. 

In summary, giving up an event of exponentially decaying probability, a random path in the Markov process has at least $Kn$ slots for Schottky loops based at some $v_{C}$. By pivoting the choice of Schottky loops at these slots, we can guarantee at least $K'n$ eventual pivots until step $n$ for some $K'>0$, outside an event of exponentially decaying probability.

Given these results, it now suffices to focus on the elements $g$ such that: \begin{enumerate}
\item $\check{g}_{1-\epsilon} \in LG$ and $d(o, \check{g}_{1-\epsilon} o) \ge (1-2\epsilon) \lambda n$,
\item $d(o, \check{g}_{\epsilon}o) \le 2\epsilon \lambda n$, and 
\item the subpath $[e, \check{g}_{\epsilon}]$ possesses at least $K' \epsilon n$ pivots for $[e, g]$.
\end{enumerate}
We then consider the equivalence class by pivoting at the first $K'\epsilon n$ pivots. By early pivoting, one can show that only few elements inside the equivalence class satisfy $\tau_{X}(g) \le (1-4 \epsilon - MK') n$, for some suitable $M>0$. By modulating $\epsilon$ and $K'$, we establish the desired result.

%
%

\medskip
\bibliographystyle{alpha}
\bibliography{pA}

\end{document}